\documentclass[journal]{IEEEtran}
\usepackage{amsfonts}
\usepackage{amsmath}
\usepackage{amsthm}
\usepackage{amssymb}
\usepackage{array}
\usepackage{xtab}
\usepackage{multirow}

\usepackage[varg]{txfonts}
\usepackage[cal=cm]{mathalpha}

\allowdisplaybreaks

\usepackage{cite}
\usepackage[caption=false]{subfig}
\usepackage{xcolor}
\usepackage[normalem]{ulem}
\usepackage{soul}
\usepackage{enumitem}
\usepackage{algorithm}
\newcommand{\commentcolor}{blue}
\usepackage[commentColor=\commentcolor]{algpseudocodex}

\usepackage{hyperref}

\usepackage{textcomp}
\newcommand{\ttilde}{\raisebox{0.5ex}{\texttildelow}}

\def\arxivversion{1}

\ifx\arxivversion\undefined

\fi

\newenvironment{sizeddisplay}[2]
 {\par\nopagebreak#1\noindent#2}
 {\nopagebreak\ignorespacesafterend}
\newcommand{\pushright}[1]{\ifmeasuring@#1\else\omit\hfill$\displaystyle#1$\fi\ignorespaces}
\ifx\arxivversion\undefined
\newcommand{\edit}[1]{\textcolor{blue}{#1}}
\else
\newcommand{\edit}[1]{#1}
\fi

\newtheorem{theorem}{Theorem}

\newcommand{\Ssub}[1]{S_{\hspace{-0.15em}#1}} 

\begin{document}
%
\title{Discrete Shortest Paths in Optimal \\Power Flow Feasible Regions}
%
%
%

\author{Daniel Turizo,~\IEEEmembership{Member,~IEEE}, Diego Cifuentes, Anton Leykin, and Daniel K. Molzahn,~\IEEEmembership{Senior Member,~IEEE}
\thanks{Daniel Turizo and Daniel K. Molzahn are with the School of Electrical and Computer Engineering, Georgia Institute of Technology, \{djturizo,molzahn\}@gatech.edu. Support from NSF contract~\#2023140.}%
\thanks{Diego Cifuentes is with the H. Milton Stewart School of Industrial and Systems Engineering, Georgia Institute of Technology, dfc3@gatech.edu.}%
\thanks{Anton Leykin is with the School of Mathematics, Georgia Institute of Technology, leykin@math.gatech.edu. Support from NSF DMS award~\#2001267.}
}%

\maketitle

\begin{abstract}
Optimal power flow (OPF) is a critical optimization problem for power systems to operate at points where cost or other operational objectives are optimized. Due to the non-convexity of the set of feasible OPF operating points, it is non-trivial to transition the power system from its current operating point to the optimal one without violating constraints. On top of that, practical considerations dictate that the transition should be achieved using a small number of small-magnitude control actions. To solve this problem, this paper proposes an algorithm for computing a transition path by framing it as a shortest path problem. This problem is formulated in terms of a discretized piece-wise linear path, where the number of pieces is fixed a priori in order to limit the number of control actions. This formulation yields a nonlinear optimization problem (NLP) with a sparse block tridiagonal structure, which we leverage by utilizing a specialized interior point method. An initial feasible path for our method is generated by solving a sequence of relaxations which are then tightened in a homotopy-like procedure. Numerical experiments illustrate the effectiveness of the algorithm.
\end{abstract}


%
\IEEEpeerreviewmaketitle

\section{Introduction}
\IEEEPARstart{T}{he} optimal power flow (OPF) is arguably the most important problem in steady state power system operation. OPF is an optimization problem that seeks to minimize an objective (usually operation cost) subject to the power flow equations governing the power system behavior and the engineering and technical constraints associated with physical operation of the system and its components \cite{crow}. A complete formulation of the OPF problem, called Alternating Current OPF (ACOPF), is a nonconvex problem with nonlinear equality constraints and hundreds to thousands of variables. 

\edit{The optimal operating point from an ACOPF solution provides values for the variables associated with controllable resources. Short-term planners and real-time system operators must determine how to transition the system from the current operating point to the optimal point.}
The control variables may be manipulated physically by, for example, controlling a floodgate in a hydro plant or the boiler in a thermal plant.
As such, the transition process between values of the controllable variables must be performed in terms of a sequence of few simple control actions, as the physical implementation limits the complexity of the execution. Furthermore, the transition between operating points should respect the system constraints in the same way that the optimal solution does. 

The problem of state transitioning in terms of few simple actions is not trivial, but some approaches have been explored in the literature. Some authors have used linear OPF approximations to tractably generate the transition as a sequence of corrective actions involving a small subset of the controllable variables. References \cite{capitanescu_limited_opf} and \cite{capitanescu_redispatch} construct a mixed-integer linear program (MILP) as an approximation to the ACOPF, while also adding hard constraints on the number of controllable variables modified. Reference \cite{sun_corrective_opf} applies sparse techniques based on high-dimensional statistics to the DCOPF formulation to generate sparse solutions with respect to a base state. These approaches, while tractable, rely on linear approximations to the original problem, so they do not guarantee that constraints are not violated during the transition. Moreover, these linear approximations improve tractability at the expense of ignoring the non-convex and possibly non-connected geometry of the feasible space \cite{lesieutre2005convexity,hill2008,bukhsh2013,hicss2014,lee_opf_convex_restriction,case9mod,molzahn_cases}. In light of these drawbacks, \cite{capitanescu_minlp} and \cite{avramidis_minlp} extend previous formulations to consider the full ACOPF, obtaining a mixed-integer nonlinear program (MINLP). These papers approximate the binary constraints in the MINLP using barrier functions, obtaining a continuous nonlinear program (NLP). This new approximation represents the original feasible set more accurately, yet still does not guarantee feasibility during the transition. 

The issue of guaranteeing feasibility during the transition process has been tackled by recent work in \cite{lee_feasible_path} and \cite{opf_sequence_path}. Reference \cite{lee_feasible_path} proposes a method for iteratively generating a sequence of convex restrictions (i.e., convex inner approximations) for the ACOPF feasible set. The sequence of sets are pairwise connected, and at some point the method generates a convex restriction containing the optimal operating point. The output of the method is a finite sequence sequence of operating points which define a piece-wise linear path connecting the current operating point and the optimal operating point. This path is guaranteed to be feasible, as it is contained in a chain of connected convex restrictions containing both operating points. Reference \cite{opf_sequence_path} proposes an algorithm for iteratively generating a sequence of feasible operating points using sensitivity information and a Newton iteration. The transition is constructed using each point in the sequence. The main drawback of approaches like those of \cite{lee_feasible_path} and \cite{opf_sequence_path} is that there is \edit{no} control over the number of intermediate operating points generated during the iteration process. That is, while these methods output a finite sequence of intermediate transition points, the length of the sequence can be arbitrarily large.

An important issue that, to the authors knowledge, has not been studied in the literature regards the \textit{amplitude}, that is, the size of of the change each variable undertakes during a control action (or equivalently, the distance between states before and \edit{after} the control action takes place). Even if the transition can be done using a few control actions involving a small number of variables, large amplitudes for these actions can be detrimental. For example, large amplitude control actions in battery energy storage systems can increase the depth-of-discharge, thus increasing battery degradation \cite{battery_degradation}. Ideally, the best transition path would be the straight line joining the current and optimal operating points since this path represents a single control action with the minimal possible amplitude. If the constraints are violated by the straight line, the transition path should be modified to avoid constraint violations, thus increasing the number and amplitude of control actions.  

This paper addresses two of the issues of operating point transitioning: the number and amplitude of control actions. \edit{We formulate the problem of minimizing the amplitude of control actions as a \textit{shortest path} problem that seeks to minimize the length of the path joining the current and optimal operating points inside the feasible space}. To this end, we propose an algorithm that computes a piece-wise linear approximation of this shortest path as a discretized path defined in terms of a chosen number of intermediate operating points. We formulate the shortest path problem as an NLP where the objective function is the path length and the optimization variables are the coordinates of the intermediate operating points, subject to the ACOPF constraints. The NLP is solved using a feasible interior point method coupled with an homotopy procedure to generate an initial feasible path. When the interior point method is applied to our formulation, the matrices involved show a sparse block tridiagonal structure. We show how to exploit this structure to reduce the interior point method's computation time, so that each iteration scales linearly with the number of intermediate operating points. We thus obtain a scalable algorithm that minimizes the amplitude of control actions and enables specifying the number of intermediate points. Numerical experiments on multiple test cases of varying sizes show the algorithm's effectiveness in finding a discretized shortest path for a specific number of points. 

\edit{
To summarize, the main contributions of our paper are:
\begin{itemize}
\item \emph{A formulation of the transitioning problem by casting it into a shortest path problem constrained to the ACOPF feasible region.} The shortest path formulation has the advantage of minimizing the amplitude of control actions. Moreover, this formulation discretizes the transition path into a finite, pre-specified number of linear pieces. Accordingly, the best transition path using exactly the desired number of control actions is obtained by solving our formulation.
\item \emph{An interior point method with sparse block tridiagonal structure for solving the shortest path problem.} The method includes a homotopy procedure for generating an initial feasible path, so no path data beyond the endpoints needs to be provided.
\item \emph{Experiments with multiple test cases of different scales}. These experiments show the method's effectiveness.
\end{itemize}
}

The rest of the paper is organized as follows. Section II describes the formulation of the ACOPF problem and the corresponding shortest path problem. Section III elaborates on the implementation of a feasible interior point method that leverages the special structure of the shortest path problem. Section IV provides a description of the complete algorithm, including a homotopy procedure for generating an initial feasible path required to execute the interior point algorithm. Section V illustrates the numerical experiments we performed. Section VI discusses conclusions and future work.

\section{Shortest Path OPF Problem Formulation}
We consider an arbitrary power system with two different operating points of interest. We wish to connect these points through a continuous path such that every point in the path is a feasible operating condition with respect to the OPF constraints. For a power system with $n$ buses, let $x \in \mathbb{R}^{2n}$ denote the real and imaginary parts of the voltage phasors for all buses, i.e., the state vector of the power system. Let $u \in \mathbb{R}^{2g}$ denote the vector of controlled variables\footnote{Usually the controlled variables of OPF problem are the voltage magnitude and active power outputs of each generator. Other type of controlled variables are valid, as long as they fit within the proposed framework.}, where $g \leq n$ is the number of generators. In particular, we denote the points we want to connect by $u_0$ and \edit{$u_{\infty} \neq u_0$}. The relationship between $x$ and $u$ is given by the power flow equations:
\begin{equation}
\label{eq:power_flow}
\begin{array}{l}
f(x,u) = [f_1(x,u), \cdots, f_{2n}(x,u)]^T = 0 \in \mathbb{R}^{2n},\\
f_k(x,u) = \begin{cases}
\frac{1}{2} x^T H_k x + r_k^T x + c_k - u_k, &  k \leq 2g, \\
\frac{1}{2} x^T H_k x + r_k^T x + c_k, & k > 2g
\end{cases}
\end{array}
\end{equation}
for appropriate symmetric matrices $H_k \in \mathbb{R}^{2n \times 2n}$ (which correspond to the $Y_k$ matrices in \cite{opf_polynomial_optimization}) and vectors $r_k \in \mathbb{R}^{2n}$. The matrices $H_k$ are highly structured: \edit{each can be written as a sum of a matrix with at most two non-zero rows and its transpose, and thus each $H_k$ has rank at most 4.}

The OPF feasible set consists of all pairs $(u,x)$ satisfying the power flow equations and the OPF constraints $g_i$ and $h_i$ (like voltage limits, line flow limits, etc.):
\begin{subequations}
\begin{align}
g_i (u) &\leq 0, \qquad i \in \mathcal{U}, \\
h_i (x) &\leq 0, \qquad i \in \mathcal{X},
\end{align}
\end{subequations}
for appropriate disjoint index sets $\mathcal{U}, \mathcal{X}$. We assume that all OPF constraints inequalities depend on either $u$ ($i \in \mathcal{U}$) or $x$ ($i \in \mathcal{X}$), but not both.\footnote{In the standard OPF problem the entries of $u$ are the generator voltage magnitudes and active power of PV buses. As such, $g_{\mathcal{U}}$ contains the voltage limits of generator buses and the active power limits of PV buses. On the other hand, $g_{\mathcal{X}}$ contains the voltage and active power limits of remaining buses, reactive power limits, line flow limits, and angle difference constraints.} The vector $x$ corresponds to the state vector associated with $u$ that satisfies \eqref{eq:power_flow}. The existence of such $x$ is not trivial, for some values $u$ there exists multiple solutions or possibly none \cite{tavora}. From the implicit function theorem \cite{iterative_solutions}, we can specify a branch of the mapping to define a continuous and injective function $\varphi$ from $u$ to $x$ in a neighborhood of $u$, as long as the Jacobian of \eqref{eq:power_flow} with respect to $x$ is non-singular in said neighborhood (see Fig. \ref{fig:opf_spaces}). We can use this information to restrict ourselves to a single branch of the mapping. Consider the pair $(u_0,x_0)$ where $u_0$ is the starting operating point and $x_0$ is the solution of \eqref{eq:power_flow} associated with $u_0$ for the branch we are interested in. Let $J(x) = \partial f(x,u)/{\partial x}$ denote the Jacobian of the power flow equations with respect to the state vector $x$ (the Jacobian with respect to $x$ is independent of $u$). If we assume that $J(x_0)$ is non-singular, then there exists a continuous and injective function $\varphi(u)$ defined by the branch of \eqref{eq:power_flow} satisfying $\varphi(u_0)=x_0$. We impose the additional constraint $u \in \mathcal{F}_0$ where $\mathcal{F}_0$ is defined as
\begin{subequations}
\begin{align}
\label{eq:feasible_set}
\mathcal{F}_0 &= \left\{{u \in \mathbb{R}^{2g} \, : \, J(\varphi(u)) \textrm{ is not singular}}\right\}, \\
\mathcal{F}_0 &= \left\{{u \in \mathbb{R}^{2g} \, : \, -| \det J(\varphi(u))| < 0}\right\}.
\end{align}
\end{subequations}
To use this formulation, we require some assumptions:
\begin{itemize}
  \item \textbf{Assumption~1:} The Jacobian $J(x_0)$ is non-singular.
  \item \textbf{Assumption~2:} The function $\varphi(u)$ can be computed.
  \item \mbox{\textbf{Assumption~3:} Both} $u_0$ and \edit{$u_{\infty}$} belong to the same connected component of $\mathcal{F}_0$.
\end{itemize}
In particular, $u_0$ and \edit{$u_{\infty}$} may be in different connected components if their associated states $x_0$ and \edit{$x_{\infty}$} belong to different branches of \eqref{eq:power_flow}. Under the previous assumptions, we can define the functions $g_i$ for all $i \in \mathcal{X}$ as
\begin{equation}
g_i(u) = h_i(\varphi(u)).
\end{equation}
so that all constraints depend only on $u$. The power flow feasible set constraint $u \in \mathcal{F}_0$ can be written analytically as
\begin{equation}
\label{eq:power_flow_constraint}
g_i(u) = -| \det J(\varphi(u))|, \quad i \in \mathcal{P},
\end{equation}
for some singleton index set $\mathcal{P}$ disjoint from $\mathcal{U},\mathcal{X}$. Define $\mathcal{I} = \mathcal{U} \cup \mathcal{X} \cup \mathcal{P}$. The interior of the power flow feasible set ($g_i, i \in \mathcal{P}$) and the OPF constraints' feasible set ($g_i, i \in \mathcal{U} \cup \mathcal{X}$) is given by all points $u \in \mathbb{R}^{2g}$ such that
\begin{equation}
\label{eq:feasible_set_implicit}
g_i(u) < 0, \quad i \in \mathcal{I}.
\end{equation}
For interior point methods, the distinction between $<$ and $\leq$ is inconsequential, as the numerical solution always lies in the interior of the feasible set.

\begin{figure}[tbp!]
    \centering
  \includegraphics[scale=0.53]{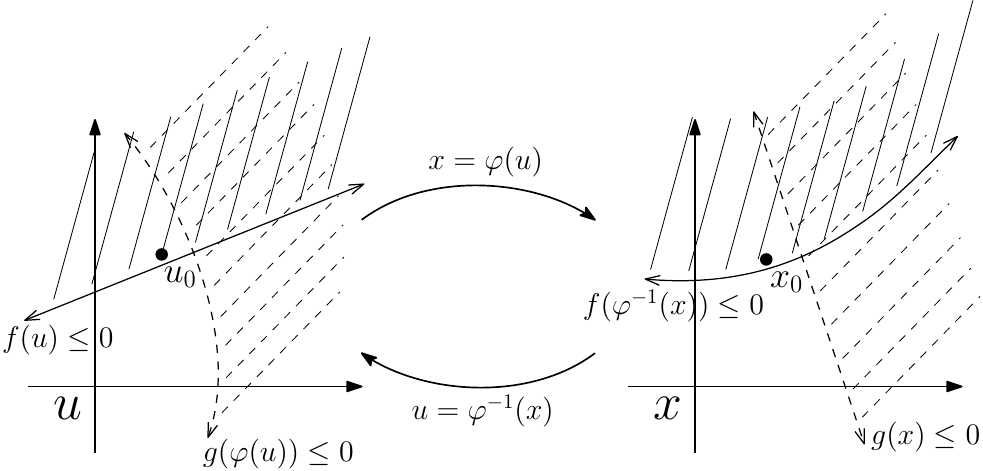}
  	\caption{Variables $u$ and $x$ in a neighborhood of $u_0$ and $x_0$ are related by the power flow mapping $\varphi$. Feasible sets generated by inequalities in $x$ can be mapped back to feasible sets in $u$ and vice-versa. As the power flow mapping $\varphi$ is nonlinear, the geometry of the mapped feasible sets will be altered.}
  	\label{fig:opf_spaces}
\end{figure}

\subsection{Optimal Control Problem}
Finding a path between two points in a set is a classical optimal control problem. If we seek the \textit{shortest} path, we then obtain an optimization problem. We define a continuation parameter $t \in [0,1]$ and the decision vector $u(t) \in C[0,1]$, where $C[0,1]$ denotes the set of continuous functions defined on the interval $[0,1]$. The shortest path problem is
\begin{equation}
\begin{array}{ll}
\inf_{u} & \int_0^1{\left({u'^T(t) u'(t)}\right)^{1/2} dt} \\
\mathrm{s.t.} & u(0) = u_0, \quad u(1) = \edit{u_{\infty}}, \\
 & g_i(u(t)) < 0, \quad \forall\; t \in [0,1], \quad i \in \mathcal{I}.
\end{array}
\label{eq:shortest_path_original}
\end{equation}
\edit{We note that the problem formulation models OPF constraints, but neglects the device dynamics involved during the transitioning process. This approximation is justified by the fact that typical power system dynamics are very fast relative to the state transitioning process. Also, this approximation is standard, as it is also adopted in previous work on related problems~\cite{capitanescu_limited_opf,capitanescu_redispatch,sun_corrective_opf,capitanescu_minlp,avramidis_minlp,lee_feasible_path,opf_sequence_path}.}

The problem described in \eqref{eq:shortest_path_original} is a calculus of variations optimization with constraints. The objective function may not be differentiable at some points (due to the square root). Moreover, problem~\eqref{eq:shortest_path_original} is naturally ill-defined, as even in the unconstrained case there are infinite gradient maps that yield a straight line between $u_0$ and \edit{$u_{\infty}$}. These issues can be avoided by requiring the gradient map to have constant norm (constant ``speed'' of transition along the path), which also simplifies the objective function. To illustrate this, assume that the path has constant norm, i.e. $\|u'(t)\| = \zeta > 0$ for all $t \in [0,1]$, then the objective function becomes
\begin{equation}
    \hspace*{-0.5em} \int_0^1{\left({u'^T (t) u'(t)}\right)^{1/2} dt} = \int_0^1{\|u'(t)\| dt} = \int_0^1{\zeta dt} = \zeta,
\end{equation}
so $\zeta$ not only denotes the ``speed'' of a particle traversing the path but also the ``time'' it takes for the particle to go from $u_0$ to \edit{$u_{\infty}$}. This formulation yields the following eikonal equation problem in terms of the arclength $\zeta$ (see \cite{viscosity_solutions,eikonal}):
\begin{equation}
\label{eq:shortest_path}
\begin{array}{ll}
\inf_{u,\zeta} & \zeta \\
\mathrm{s.t.} & u(0) = u_0, \quad u(1) = \edit{u_{\infty}}, \\
 & \|u'(t)\| = \zeta, \quad \forall\; t \in [0,1], \\
 & g_i(u(t)) < 0, \quad i \in \mathcal{I}.
\end{array}
\end{equation}
Any numerical approach to solving this problem must honor the feasible set constraints in \eqref{eq:feasible_set_implicit}, as there does not exist a state vector $x$ associated with any $u \notin \mathcal{F}_0$. Also, there is no trivial feasible starting path available in general. To circumvent this issue, we next propose a discretized version of the problem.

\subsection{Piece-wise Linear Path Approximation}
We restrict the search space from $C[0,1]$ to the space of piece-wise linear paths $PL[0,1]$.\footnote{Note that $PL[0,1]$ is dense in $C[0,1]$ with respect to the uniform norm, as the Schauder system of C[0,1] is composed of piece-wise linear functions~\cite{heil}.}
More specifically, we will consider the space of piece-wise linear paths with $K+1$ pieces, $PL_{K+1}[0,1]$. Let the characteristic (sometimes called indicator) function $\chi_E (t)$ be defined as
\begin{equation}
    \chi_E(t) = \left\{\begin{array}{cc}
        1 & t \in E, \\
        0 & t \notin E.
    \end{array}
    \right.
\end{equation}
We consider a piece-wise linear path $u(t) \in PL_{K+1}[0,1]$ defined by $K+2$ points $\{u_k\}_{k=0}^{K+1}$ and parameters $\{t_k\}_{k=0}^{K+1}$:
\begin{equation}
\label{eq:p_t}
\begin{array}{l}
u(t) = u_0 \chi_{\{0\}}(t) + \sum_{k=1}^{K+1}{c_k(t) \chi_{\left({t_{k-1}, t_k}\right]}(t)} ,\\
\text{with }c_k(t) = u_{k-1} + (u_k - u_{k-1}) \frac{t - t_{k-1}}{t_k - t_{k-1}}.
\end{array}
\end{equation}
The parameter values $t_k$ satisfy
\begin{equation}
\label{eq:path_spacing}
t_0 = 0 < t_1 < \cdots < t_K < t_{K+1} = 1.
\end{equation}
\edit{Additionally, we require that $u_{K+1} = u_{\infty}$, as the path endpoints must match the desired endpoints. For convenience, from now on we shall write $u_{K+1}$ and $u_{\infty}$ interchangeably.} We want to compute the path $p(t)$ that minimizes the objective function. Note that, for fixed values $\{t_k\}_{k=0}^{K+1}$, $u(t) \in PL_{K+1}[0,1]$ can be identified with $\{u_k\}_{k=0}^{K+1}$. Thus, the control problem reduces to computing the points $\{u_k\}_{k=1}^{K}$ that minimize the objective (recall that $u_0$ and $u_{K+1}$ are the path endpoints, and thus they are known):
\begin{equation}
\label{eq:piecewise_optim}
\begin{array}{ll}
\inf_{u_1,\cdots,u_K, \zeta} & \zeta \\
\mathrm{s.t.} & \|u'(t)\| = \zeta, \quad \forall\; t \in [0,1], \\
 & g_i(u(t)) < 0, \quad i \in \mathcal{I}.
\end{array}
\end{equation}
We concatenate the points $\{u_k\}_{k=1}^{K}$ into a single vector $\vec{u} = [u_1^T,\cdots,u_K^T]^T \in \mathbb{R}^{2gK}$. Replacing \eqref{eq:p_t} in \eqref{eq:piecewise_optim} yields
\begin{alignat}{2}
&\inf_{\vec{u}, \zeta}\quad && \zeta \nonumber \\
&\,\mathrm{s.t.} && c_k(\vec{u},t) = u_{k-1} + (u_k - u_{k-1}) \frac{t - t_{k-1}}{t_k - t_{k-1}}, \nonumber \\
& && \|c'_k(\vec{u},\tau)\| = \zeta, \quad \forall \tau \in \left({t_{k-1}, t_k}\right], \nonumber \\
 & && g_j(c_k(\vec{u},\tau)) < 0, \quad j \in \mathcal{I}, \nonumber \\
 & && \forall\; t \in [0,1], \quad k=1,\ldots,K+1 .
\end{alignat}
The optimization problem is now finite dimensional, yet the constraints are still infinite dimensional. For the purpose of tractability, we will relax the constraints by only enforcing them at the corner points $u_k$. This means that the path may violate constraints in between corner points. However, if needed, we can add more discretization points to mitigate this issue. As each piece of the path is linear, the infinite-dimensional constant speed constraint is equivalent to the finite dimensional constraint that enforces the slopes of each piece of the path to be equal in norm. Also note that the constant speed constraint implies that $\zeta \geq 0$, so minimizing $\zeta$ is equivalent to minimizing $\zeta^2$. These changes yield the following problem:
\begin{alignat}{2}
\label{eq:point_optim}
&\inf_{\vec{u}, \zeta}\quad && \zeta^2 \nonumber \\
&\,\mathrm{s.t.} && \frac{\|u_k - u_{k-1}\|}{t_k - t_{k-1}} = \zeta, \quad k=1,\ldots,K+1, \nonumber \\
 & && g_j(u_i) < 0, \quad j \in \mathcal{I}.
\end{alignat}
The norm constraints are nonlinear inequalities, and hence are non-convex. Any solution method for this problem should be able to at least converge to a local optimum, even in the presence of non-convexities. To this end, we will reformulate the problem in a way that is advantageous for the numerical method we will use. Define the constants
\begin{equation}
    w_k = \frac{1}{(t_k - t_{k-1})^2 \|u_{K+1} - u_0\|^2} > 0,
\end{equation}
for $k=1,\ldots,K+1$. Then \eqref{eq:point_optim} is equivalent to
\begin{alignat}{2}
&\inf_{\vec{u}} \quad && \frac{1}{K+1} \sum_{k=1}^{K+1}{w_k \|u_k - u_{k-1}\|^2} \nonumber \\
&\,\mathrm{s.t.} && w_{i+1} \|u_{i+1} - u_i\|^2 = w_i \|u_i - u_{i-1}\|^2, \quad i=1,\ldots,K, \nonumber \\
 & && g_j(u_i) < 0, \quad j \in \mathcal{I}.
\end{alignat}
In this formulation, the Hessian of the objective function is positive definite, which will prove useful for the interior point iteration described in the next section.

\section{Log-Barrier Newton Method Implementation}
The discretized shortest path problem in \eqref{eq:point_optim_numerical} has a  tridiagonal structure which is not leveraged by standard interior point solvers. For this reason, we developed \edit{a} specialized interior point implementation that makes use of the problem structure to reduce the computational complexity of solving the problem. This section is dedicated to explaining in detail the core iterative process behind this specialized solver.

\subsection{Interior Point Iteration}
\edit{
The shortest path problem has a parallel structure since the constraints do not depend on the full decision vector $\vec{u}$, but only on their associated point $u_i$. The only source of coupling between points comes from the objective and the equality constraints, which both have simple block-tridiagonal structures that can be exploited by a specialized interior point iteration. To reduce the computation time even further, we will use an extended formulation of the optimization problem including the system state $x_i$ associated with each control vector $u_i$ via the power flow equations. Specifically, we define
\begin{equation}\label{eq:p}
p = [p_1^T, \cdots, p_K^T]^T, \quad p_i = [u_i^T, x_i^T]^T, \quad i = 1 , \ldots, K.
\end{equation}
Now we write the optimization problem as
\begin{subequations}
\label{eq:point_optim_numerical}
\begin{alignat}{2}
&\inf_{p} \quad && \frac{1}{K+1} \sum_{k=1}^{K+1}{w_k \|u_k - u_{k-1}\|^2} \\
&\,\mathrm{s.t.} && w_{i+1} \|u_{i+1} - u_i\|^2 = w_i \|u_i - u_{i-1}\|^2, \quad i=1,\ldots,K, \label{eq:speed_constraint} \\
 & && f(u_i, x_i) = 0, \\
 & && g_j(p_i) < 0, \quad j \in \mathcal{I}.
\end{alignat}
\end{subequations}
This formulation is larger in terms of variables and has more equality constraints due to the power flow equations. However, the OPF constraints written in terms of $u_i$ and $x_i$ have extremely sparse formulations.
}

The log-barrier formulation that is central to the interior point method embeds the inequality constraints into the objective and then numerically solves the first-order Karush-Kuhn-Tucker (KKT) equations. We define the index set $\mathcal{E}=\{1,\ldots,K\}$ and the functions $c_i$ as
\begin{equation}
c_i(p) = w_i \|u_i - u_{i-1}\|^2 - w_{i+1} \|u_{i+1} - u_i\|^2, \quad i \in \mathcal{E}.
\end{equation}
We also define the objective as
\begin{equation}
\phi(p) = \frac{1}{K+1} \sum_{k=1}^{K+1}{w_k \|u_k - u_{k-1}\|^2}.
\end{equation}
Define $f(p_i)= f(u_i, x_i)$, so we can reformulate \eqref{eq:point_optim_numerical} using a barrier parameter $\mu>0$ as
\begin{alignat}{2}
&\inf_{p,s} \quad && \phi(p) - \mu \sum_{i=1}^{K}\sum_{j \in \mathcal{I}}{\ln[(s)_{|\mathcal{I}|(i-1)+j}]} \nonumber \\
&\,\mathrm{s.t.} && f(p_i) = 0, \quad i=1,\ldots,K, \nonumber \\
 & && c_j(p) = 0, \quad j \in \mathcal{E}, \nonumber \\
 & && g_j(p_i) + (s)_{|\mathcal{I}|(i-1)+j} = 0, \quad j \in \mathcal{I}, \label{eq:barrier_problem}
\end{alignat}
where $s$ is a vector of size $K|\mathcal{I}|$ and $(s)_k$ denotes the $k$-th entry of $s$. Note that this formulation is only equivalent to \eqref{eq:point_optim_numerical} when all the entries of $s$ are strictly positive. However, such constraints are unnecessary, as the logarithmic terms act as a barrier preventing the entries of $s$ from becoming non-positive. Define $g_{\mathcal{I}}(p_i)$ as the vector of inequality constraints (evaluated at a particular point on the path), $c_{\mathcal{E}}(p)$ as the vector of equality constraints, and let $D_p$ be the Jacobian operator (with respect to $p$). Let $v \in \mathbb{R}^{2nK}$, $y \in \mathbb{R}^{|\mathcal{E}|}$, and $z \in \mathbb{R}^{K|\mathcal{I}|}$ be vectors of Lagrange multipliers of the power flow, equality, and inequality constraints, respectively. Specifically, we write
\begin{equation}
v^T = [v^T_1, \ldots, v^T_K], \quad z^T = [z^T_1, \ldots, z^T_K],
\end{equation}
where $v_i \in \mathbb{R}^{2n}$ and $z_i \in \mathbb{R}^{|\mathcal{I}|}$ are the vectors of Largrange multipliers associated with power flow and inequality constraints evaluated at $p_i$, for $i = 1 , \ldots, K$. The stationarity condition, split for the derivatives with respect to $p$ and $s$, is
\begin{subequations}
\begin{align}
0 &= \nabla_p \phi(p) + \sum_{i=1}^{K}{[D_{p_i} f(p_i)]^T v_i} + [D_p c_{\mathcal{E}}(p)]^T y \nonumber \\
&\quad + \sum_{i=1}^{K}{[D_{p_i} g_{\mathcal{I}}(p_i)]^T z_i}, \\
0 &= -(\mu \vec{1}) \oslash s + z,
\end{align}
\end{subequations}
where $\oslash$ denotes element-wise division, $\vec{1}$ is a vector of ones, and $\nabla_p$($D_p$) denotes the gradient (Jacobian) with respect to $p$. More explicitly, the gradient term is
\begin{equation}
    \nabla_p \phi(p) = \left[\frac{\partial \phi(p)}{\partial p_i}\right]^K_{i=1},
\end{equation}
where the notation $[\cdot]^K_{i=1}$ indicates vertical concatenation of scalars/vectors/matrices indexed by $i$, along the ordered set ${1,\ldots,K}$. In the same fashion, we can write the Jacobian as
\begin{equation}
    D_p c_{\mathcal{E}}(p) = \left[ \nabla^T_p c_i(p)\right]_{i\in\mathcal{E}}.
\end{equation}
Define the vectors $d_k$ as
\begin{equation}
d_k = [(u_k - u_{k-1})^T, 0_{1 \times 2n}]^T, \quad k=1,\ldots, K+1,
\end{equation}
then we have that
\begin{equation}
\frac{\partial \phi(p)}{\partial p^T_i} = 2w_i d_i - 2w_{i+1} d_{i+1}, \quad i=1,\ldots, K.
\end{equation}
We also notice that
\begin{subequations}
\begin{align}
D_p f(p_i) &= [D_{p_1} f(p_i), \dots, D_{p_K} f(p_i)], \\
D_p f(p_i) &= [0_{2n \times 2(g+n)(i-1)}, D_{p_i} f(p_i), 0_{2n \times 2(g+n)(K-i)}],
\end{align}
\end{subequations}
and similarly
\begin{subequations}
\begin{align}
D_p g_{\mathcal{I}}(p_i) &= [D_{p_1} g_{\mathcal{I}}(p_i), \dots, D_{p_K} g_{\mathcal{I}}(p_i)], \\
D_p g_{\mathcal{I}}(p_i) &= [0_{|\mathcal{I}| \times 2(g+n)(i-1)}, D_{p_i} g_{\mathcal{I}}(p_i), 0_{|\mathcal{I}| \times 2(g+n)(K-i)}].
\end{align}
\end{subequations}
As a consequence, the stationarity condition becomes
\begin{subequations}
\begin{align}
0 &= \nabla_p \phi(p) + ([D_p f(p_i)]^K_{i=1})^T v + [D_p c_{\mathcal{E}}(p)]^T y \nonumber \\
&\quad + ([D_p g_{\mathcal{I}}(p_i)]^K_{i=1})^T z, \\
0 &= -(\mu \vec{1}) \oslash s + z.
\end{align}
\end{subequations}
The stationarity condition combined with the equality constraints define a set of nonlinear equations that can be solved numerically to find a KKT point. The Lagrangian of the problem, excluding the barrier terms, is
\begin{align}
L(p,s,v,y,z) &= \phi(p) + v^T[f(p_i)]^K_{i=1} + y^T c_{\mathcal{E}}(p) \nonumber \\
&\quad + z^T ([g_{\mathcal{I}}(p_i)]^K_{i=1} + s),
\end{align}
so the first-order KKT conditions can be written as
\begin{subequations}
\begin{align}
0 &= \nabla_p L(p,s,v,y,z) , \\
0 &= s \circ z - \mu \vec{1}, \label{eq:KKT_sz} \\
0 &= f(p_i), \quad i=1,\ldots,K, \\
0 &= c_{\mathcal{E}}(p), \\
0 &= g_{\mathcal{I}}(p_i) + (s)_{(|\mathcal{I}|(i-1)+1):|\mathcal{I}|i}, \quad i=1,\ldots,K,
\end{align}
\end{subequations}
where $\circ$ denotes element-wise multiplication. Notice that \eqref{eq:KKT_sz} implies that the entries of $s$ and $z$ must have the same sign, so $z$ must also have positive entries. For brevity, we will rename some vectors and matrices as
\begin{equation}
\label{eq:matrix_blocks}
\begin{array}{ll}
\Sigma = {\rm diag}(z \oslash s), & 
c_{\mathcal{I}}(p) = [g_{\mathcal{I}}(p_i)]^K_{i=1}, \\[0.35em]
D_{\mathcal{E}} = D_p c_{\mathcal{E}}, &
D_{\mathcal{I}} = [D_p g_{\mathcal{I}}(p_i)]^K_{i=1}, \\[0.35em]
f(p) = [f(p_i)]^K_{i=1} & D_f = [D_p f(p_i)]^K_{i=1}.
\end{array}
\end{equation}
Applying Newton's method, and omitting dependencies for brevity, we obtain the following update equation:
\begin{equation}
\label{eq:newton_step_full}
J \left[{\begin{matrix}
 \Delta p \\
 \Delta s \\
 \Delta v \\
 \Delta y \\
 \Delta z 
\end{matrix}}\right] = -\left[{\begin{matrix}
 \nabla_p L \\
 z - (\mu \vec{1}) \oslash s \\
 f(p) \\
 c_{\mathcal{E}} \\
 c_{\mathcal{I}}(p) + s
\end{matrix}}\right], \quad
J = \left[{\begin{matrix}
 \nabla^2_{pp} L & 0 & D_f^T & D^T_{\mathcal{E}} & D^T_{\mathcal{I}} \\
 0 & \Sigma & 0 & 0 & I \\
 D_f & 0 & 0 & 0 & 0 \\
 D_{\mathcal{E}} & 0 & 0 & 0 & 0 \\
 D_{\mathcal{I}} & I & 0 & 0 & 0
\end{matrix}}\right], 
\end{equation}
where the second row block has been left-multiplied by ${\rm diag}(s)^{-1}$ to make the matrix symmetric. The rank of the Newton matrix depends directly on $\nabla^2_{pp} L$, $D_f$, and $D_{\mathcal{E}}$. More specifically, if $\nabla^2_{pp} L$, $D_f$, and $D_{\mathcal{E}}$ are full rank, then the Newton matrix is invertible. Note that $D_f$ is a block diagonal matrix comprised of the power flow Jacobian at each $p_i$. This means that $D_f$ is full-rank, as the power flow feasibility constraint guarantees the invertibility of the square block of the power flow Jacobian associated with $x_i$. Hence, we only need to concern ourselves with studying the ranks of $\nabla^2_{pp} L$ and $D_{\mathcal{E}}$. We will first provide conditions under which $D_{\mathcal{E}}$ has full rank. Recall that
\begin{subequations}
\begin{align}
D_{\mathcal{E}} &= D_p c_{\mathcal{E}} = [D_{p_1} c_{\mathcal{E}}, \dots, D_{p_K} c_{\mathcal{E}}], \\
D_{\mathcal{E}} &= \left[\frac{\partial c_j}{\partial p^T_i}\right]_{j \in \mathcal{E}}, \quad i = 1,\ldots,K, \\
\frac{\partial c_j}{\partial p^T_i} &= 
\left\{ {\begin{array}{ll}
   -2 w_j d^T_j, & i=j-1 \\
   2 w_j d^T_j + 2 w_{j+1} d^T_{j+1}, & i=j \\
   -2 w_{j+1} d^T_{j+1}, & i=j+1 \\
   0, & {\rm else} \\
\end{array}} \right. ,
\end{align}
\end{subequations}
so $D_{\mathcal{E}}$ is a $K \times 2(g+n)K$ block tridiagonal matrix, with blocks of size $1 \times 2(g+n)$. Next, we prove the claim.
\begin{theorem}\label{thm:DE_rank}
For any $j=1,\ldots,K+1$ define $b_j(p) = w_j d_j$ and assume that $b_j (p)\neq 0$ for all $j=1,\ldots,K+1$ (this is true if and only if $d_j \neq 0$). Let $q_j$ be
\begin{equation}
q_j(p) = b_j / b^T_j b_j, \qquad \forall j = 1, \ldots, K+1 .
\end{equation}
If $\sum^{K+1}_{j=1} q_j(p) \neq 0$, then $D_{\mathcal{E}}$ is full rank.
\end{theorem}
\begin{proof}[Proof]
See Appendix C{\ifx\arxivversion\undefined{ of the extended version of this paper in~\cite{shortest_path_arxiv}}\fi}.
\end{proof}

From this result, we can guarantee that $D_{\mathcal{E}}$ is full rank as long as we prevent any $d_j$ from becoming $0$. We also need to safeguard the algorithm against cases where $\sum^{K+1}_{j=1} q_k = 0$. This condition is a vector generalization of the condition of impedance loops not adding to zero in order to guarantee the invertibility of the admittance matrix in transmission systems (see \cite{turizo}). We propose a simple step rejection procedure as a safeguard; this is detailed in Appendix B{\ifx\arxivversion\undefined{ of the extended version of this paper in~\cite{shortest_path_arxiv}}\fi}.

The Lagrangian Hessian, $\nabla^2_{pp} L$, may not be invertible, but it is a very structured matrix. In Appendix A{\ifx\arxivversion\undefined{ of the extended version of this paper in~\cite{shortest_path_arxiv}}\fi}, we show that $\nabla^2_{pp} L$, $\nabla^2_{pp} \phi$, and $\nabla^2_{pp} (y^T c_\mathcal{E})$ are symmetric and block tridiagonal, $\nabla^2_{pp} (L - \phi - y^T c_\mathcal{E})$ is block diagonal (with block sizes $2(g+n) \times 2(g+n)$ for all matrices), and $\nabla^2_{pp} L \succeq 0$. Invertibility and other issues (like indefiniteness) can be easily corrected by leveraging the block structure of $\nabla^2_{pp} L$ and its components, as shown in the next subsection.

\subsection{Newton Step Correction}
\edit{
To ensure that the Newton step in the primal variables, $\Delta p$, yields a descent direction, we require $\nabla^2_{pp} L$ to be positive definite in the tangent space of the linearized constraints. The simplest way to satisfy the condition is to modify $\nabla^2_{pp} L$ to make it positive definite. To this end, notice that
\begin{equation}
\nabla^2_{pp} L = \nabla^2_{pp} \phi + \nabla^2_{pp}(L - \phi).
\end{equation}
We already know that $\nabla^2_{pp} \phi \succeq 0$, so any source of indefiniteness must come from the Lagrangian terms of the constraints, $L - \phi$. We modify the Hessian by adding to it a diagonal matrix $S$ such that $\nabla^2_{pp}(L - \phi) + S \succ 0$. A strategy for selecting $S$ with low computational cost is discussed in Appendix B1{\ifx\arxivversion\undefined{ of the extended version of this paper in~\cite{shortest_path_arxiv}}\fi}. The Newton step with Hessian correction becomes
\begin{sizeddisplay}{\small}{}
\begin{equation}
\label{eq:newton_step_corrected}
J' \left[{\begin{matrix}
 \Delta p \\
 \Delta s \\
 \Delta v \\
 \Delta y \\
 \Delta z 
\end{matrix}}\right] = -\left[{\begin{matrix}
 \nabla_p L \\
 z - (\mu \vec{1}) \oslash s \\
 f(p) \\
 c_{\mathcal{E}} \\
 c_{\mathcal{I}}(p) + s
\end{matrix}}\right], \quad
J' = \left[{\begin{matrix}
 \nabla^2_{pp} L + S & 0 & D_f^T & D^T_{\mathcal{E}} & D^T_{\mathcal{I}} \\
 0 & \Sigma & 0 & 0 & I \\
 D_f & 0 & 0 & 0 & 0 \\
 D_{\mathcal{E}} & 0 & 0 & 0 & 0 \\
 D_{\mathcal{I}} & I & 0 & 0 & 0
\end{matrix}}\right], 
\end{equation}
\end{sizeddisplay}
where $S=0$ if it is determined that no correction is needed. Otherwise, $S$ is chosen as described in Appendix B1{\ifx\arxivversion\undefined{ of the extended version of this paper in~\cite{shortest_path_arxiv}}\fi}. A procedure for determining whether the Hessian needs correction or not is discussed in a later subsection.
}

\subsection{Newton Step Permutation}
\edit{
The Newton step computation requires solving the linear system \eqref{eq:newton_step_full}, which has size $K(2g + 4n + 2|\mathcal{I}| + 1)$, so a matrix factorization requires $O(K^3(n+|\mathcal{I}|)^3)$ operations if the matrix is dense. Fortunately, the Newton matrix $J$ is sparse, so the linear system can be solved much more quickly using a sparse linear solver. The performance a sparse solver depends on the amount of extra entries filled during the factorization step, which in turn depends on the specific input matrix. In particular, matrices with low bandwidth\footnote{\edit{The bandwidth of a symmetric $n \times n$ matrix $A$ is the smallest integer $b \geq 0$, if it exists, such that $(A)_{ij} = 0$ for all $i=1,\ldots,n$ and $j=i+b+1,\ldots,n$. If $b$ does not exist, then the bandwidth is $n-1$.}}usually generate very little fill-in during factorization. As a consequence, some sparse solvers employ techniques like the reverse Cuthill-McKee (RCM) algorithm to generate a permuted matrix with reduced bandwidth \cite{rcm_distributed}. However, the minimum bandwidth permutation of some sparse matrices may still be very large. Even if there exists a low bandwidth permutation for the matrix, techniques like RCM are heuristic in nature, so they are not guaranteed to achieve a significant bandwidth reduction\footnote{\edit{The bandwidth minimization problem for symmetric $n \times n$ matrices is known to be NP-hard (see problem [GT40] in Appendix A1 of \cite{np_hardness}). It is also NP-hard, for any $\epsilon > 0$, to approximate the minimum bandwidth to a factor of $3/2 - \epsilon$ (see \cite{bandwidth_approx_incractability}).}}. 
}

\edit{
We will show by construction that there exists a permutation of the Newton matrix that makes it block tridiagonal, with square blocks of size $2g + 4n + 2|\mathcal{I}| + 1$ and bandwidth at most $4g + 4n + 2|\mathcal{I}| + 1$ (in particular, this means that the cost of solving \eqref{eq:newton_step_full} scales linearly with $K$). To this end, we recall that $D_{\mathcal{I}}$, $D_f$, and $\nabla_{pp}^2(L - \phi - y^T c_{\mathcal{E}})$ are block diagonal; on the other hand $D_{\mathcal{E}}$, $\nabla_{pp}^2\phi$, and $\nabla_{pp}^2y^T c_{\mathcal{E}}$ are block tridiagonal. Let $I_i$ denote the $i \times i$ identity matrix for any $i \in \mathbb{N}$. We define the permutation matrix $P$ as
\begin{subequations}
\begin{align}
P &= [P_i]_{i=1}^K, \quad P_i = [P_{ij}]_{j=1}^5, \\
P_{i1} &= \big[
  0_{1 \times K(2g + 4n + |\mathcal{I}|)}, 0_{1 \times (i-1)}, 1, 0_{1 \times (K - i)}, 0_{1 \times K |\mathcal{I}|}\big], \\
P_{i2} &= \big[
  0_{2(g+n) \times (i-1)2(g+n)}, I_{2(g+n)}, 0_{2(g+n) \times (K - i)2(g+n)}, \nonumber \\
  &\quad\quad 0_{2(g+n) \times K(2|\mathcal{I}| + 2n + 1)}\big], \\
P_{i3} &= \big[
  0_{|\mathcal{I}| \times 2K(g+n)}, 0_{|\mathcal{I}| \times (i-1)|\mathcal{I}|}, I_{|\mathcal{I}|}, 0_{|\mathcal{I}| \times (K-i)|\mathcal{I}|}, \nonumber \\
  &\quad\quad 0_{|\mathcal{I}| \times K(2n + |\mathcal{I}| + 1)}\big], \\
P_{i4} &= \big[
  0_{2n \times K(2g+2n+|\mathcal{I}|)}, 0_{2n \times (i-1)2n}, I_{2n}, 0_{2n \times (K-i)2n}, \nonumber \\
  &\quad\quad 0_{2n \times K(|\mathcal{I}| + 1)}\big], \\
P_{i5} &= \big[
  0_{|\mathcal{I}| \times K(2g+4n+|\mathcal{I}|+1)}, 0_{|\mathcal{I}| \times (i-1)|\mathcal{I}|}, I_{|\mathcal{I}|}, 0_{|\mathcal{I}| \times (K-i)|\mathcal{I}|}\big].
\end{align}
\end{subequations}
Next we define $\Sigma_i = {\rm diag}(z_i \oslash s_i)$ for $i=1\ldots,K$ and we split the correction matrix $S$ into blocks as
\begin{equation}
S = \left[\begin{matrix}
\Ssub{1} & & \\
 & \ddots & \\
 & & \Ssub{K}
\end{matrix}\right], \quad \Ssub{i} \in \mathbb{R}^{2(g+n)}, \;i=1,\ldots,K. \label{eq:block_perturbation}
\end{equation}
Hence, by direct computation, we obtain that
\begin{equation}
P J' P^T = \begin{bmatrix}
\Phi_{1,1} & -\Phi_{2,1}^T &  &   \\
-\Phi_{2,1} & \ddots & \ddots &   \\
 &  \ddots & \ddots & -\Phi_{2,K-1}^T \\
 &   & -\Phi_{2,K-1} & \Phi_{1,K}
\end{bmatrix} ,
\end{equation}
where $\Phi_{1,i}$ and $\Phi_{2,j}$ are square matrices of size $2g + 4n + 2|\mathcal{I}| + 1$, defined as
\begin{subequations}
\begin{align}
\Phi_{1,i} &= \begin{bmatrix}
 0 & D_{p_i}c_i(p) & 0 & 0 & 0  \\
 D_{p_i}^T c_i(p) & \nabla^2_{p_i p_i} L + \Ssub{i} & 0 & D_{p_i}^T f(p_i) & D_{p_i}^T g_{\mathcal{I}}(p_i) \\
 0 & 0 & \Sigma_i & 0 & I \\
 0 & D_{p_i}f(p_i) & 0 & 0 & 0 \\
 0 & D_{p_i} g_{\mathcal{I}}(p_i) & I & 0 & 0
\end{bmatrix}, \\
\Phi_{2,j} &= \begin{bmatrix}
 0 & D_{p_j}c_{j+1}(p) & 0 & 0 & 0  \\
 D_{p_j}^T c_{j+1}(p) & \nabla^2_{p_{j+1} p_j}(\phi + y^T c_{\mathcal{E}}) & 0 & 0 & 0 \\
 0 & 0 & 0 & 0 & 0 \\
 0 & 0 & 0 & 0 & 0 \\
 0 & 0 & 0 & 0 & 0
\end{bmatrix},
\end{align}
\end{subequations}
for $i=1\ldots,K$ and $j=1\ldots,K-1$. Recall that neither $c_{\mathcal{E}}$ nor $\phi$ depend on any $x_i$, hence
\begin{subequations}
\begin{align}
D_{p_j}c_{j+1}(p) &= \begin{bmatrix}
 D_{u_j}c_{j+1}(p) & 0
\end{bmatrix}, \\
\nabla^2_{p_{j+1} p_j}(\phi + y^T c_{\mathcal{E}}) &= \begin{bmatrix}
 \nabla^2_{u_{j+1} u_j}(\phi + y^T c_{\mathcal{E}}) & 0 \\
 0 & 0
\end{bmatrix}.
\end{align}
\end{subequations}
Therefore we can write $\Phi_{2,j}$ as
\begin{equation}
\Phi_{2,j} = \begin{bmatrix}
 0 & D_{u_j}c_{j+1}(p) & 0 \\
 D_{u_j}^T c_{j+1}(p) & \nabla^2_{u_{j+1} u_j}(\phi + y^T c_{\mathcal{E}}) & 0 \\
 0 & 0 & 0 \\
\end{bmatrix}, 
\end{equation}
for $j=1\ldots,K-1$. From this representation it is clear that the bandwidth of $P J' P^T$ is at most $4g + 4n + 2|\mathcal{I}| + 1$. Finally, we use the fact that $P^{-1}=P^T$ for any permutation matrix to compute the Newton step by solving
\begin{equation}\label{eq:newton_step_permuted}
(PJ'P^T) \xi = -P\left[{\begin{matrix}
 \nabla_p L \\
 z - (\mu \vec{1}) \oslash s \\
 f(p) \\
 c_{\mathcal{E}} \\
 c_{\mathcal{I}}(p) + s
\end{matrix}}\right], \quad
 \xi = P\left[{\begin{matrix}
 \Delta p \\
 \Delta s \\
 \Delta v \\
 \Delta y \\
 \Delta z 
\end{matrix}}\right].
\end{equation}
Notice that $P$ is constant across iterations, so it only needs to be computed once.
}

\edit{
The Jacobian structure illustrated by this permutation is very amenable to parallelism. On one hand each $\Phi_{2,j}$ has at most $6g$ non-zero entries, so they can be computed in $O(g)$ time. On the other hand, each of $\Phi_{1,i}$, $\nabla_{p_k}L$, and $c_{\mathcal{I}}(p_k)$ can be computed in parallel, so the cost of the Newton step scales linearly with the number points divided by the number of parallel workers.
}

\subsection{Newton Iteration Algorithm}
Thus far, we have detailed a procedure for computing the Newton step in an interior point iteration for solving~\eqref{eq:point_optim_numerical}. However, a robust implementation must also incorporate safeguards for issues related to strong non-linearity, indefiniteness, strict positivity of dual variables, and scale disparity between primal and dual variables. We discuss these issues and their solutions, including a procedure for determining if the Hessian needs correction, in Appendix B{\ifx\arxivversion\undefined{ of the extended version of this paper in~\cite{shortest_path_arxiv}}\fi}. Once a complete Newton iteration for the interior point method is implemented, we can solve the barrier problem for a fixed barrier parameter $\mu$, as long as we are provided an initial feasible path. Pseudo-code of the procedure given an initial feasible path $p$ is described in Appendix B4{\ifx\arxivversion\undefined{ of the extended version of this paper in~\cite{shortest_path_arxiv}}\fi}.

\section{Initial Feasible Path Generation}
The last missing part of the full algorithm is a procedure for generating an initial feasible path. In the unconstrained case, the straight line connecting $u_0$ to \edit{$u_{K+1}$} is a feasible path (and, in fact, the shortest one). To include the effect of constraints, we introduce a homotopy-like procedure: we start with a relaxed version of the problem where the straight line is feasible and then we solve increasingly tighter relaxations until the original problem is recovered. A way to interpret this procedure is to consider the constraints as continuously pushing and deforming the straight line until a curved feasible path is obtained. If the transition problem is infeasible ($u_0$ and \edit{$u_{K+1}$} lie in different connected components of the feasible region), then at some point of the homotopy some constraints will try to cut the path to get each piece to a different connected component. If the path's corners are too close, such a transformation of the path would violate the constant speed constraint \eqref{eq:speed_constraint} and the homotopy would fail (see Fig. \ref{fig:case9mod}).

We next formally describe the path generation procedure. First, we notice that the power flow feasibility constraint ($g_i,i \in \mathcal{P}$, see \eqref{eq:power_flow_constraint}) is a special case as it is not differentiable on its boundary. This means that there exists no differentiable relaxation of it. Nevertheless, the power flow feasible region (i.e., the set of power injections for which a power flow solution exists) is typically much larger than the OPF constraints' feasible region, so we can thus assume that the straight line (in the space of control variables) does not violate the power flow feasibility constraint:
\begin{itemize}
  \item \textbf{Assumption~4:} The straight line joining $u_0$ and \edit{$u_{K+1}$} is contained in the power flow feasibility set $\mathcal{F}_0$.
\end{itemize}
Under Assumption~4, we do not need to relax the power flow feasibility constraint during the homotopy process. The homotopy procedure for addressing the remaining constraints is relatively simple. Assume that the user provides a path spacing $\{t_k\}_{k=0}^{K+1}$ satisfying \eqref{eq:path_spacing}. \edit{Let $p$ be the current candidate path. At the start of the procedure, our candidate path will be a straight line when projected to the space of control variables, so it satisfies
\begin{equation}
\label{eq:line_path_u}
    u_k = u_0 + t_k (u_{K+1} - u_0), \qquad k = 0,\ldots,K+1.
\end{equation}
The corresponding $x_k$ are computed by solving the power flow equations, that is
\begin{equation}
\label{eq:line_path_x}
f(u_k , x_k) = 0, \qquad k = 0,\ldots,K+1,
\end{equation}
and the candidate path is formed by applying \eqref{eq:p}.
Next we compute the relaxation parameter $\beta$ as the maximum violation of any constraint across all path corners (excluding the endpoints), multiplied by a margin $\kappa_{\beta} > 1$:
\begin{equation}
\label{eq:relaxation_vector}
    \beta = \max_{\begin{subarray}{c}i=1,\ldots,K \\ j \in \mathcal{I}\end{subarray}} \left( g_j(p_i) \right).
\end{equation}
The vector of relaxed constraints, $g_{\beta,\mathcal{I}}$, is defined for any $j \in \mathcal{I}$ and any $p_i$ as
\begin{equation}
\label{eq:relaxed_constraints}
    \left(g_{\beta,\mathcal{I}}(p_i)\right)_j = \left\{ {\begin{array}{ll}
   g_j(p_i), & j \in \mathcal{P} \\
   g_j(p_i) - \kappa_{\beta} \beta, & {\rm else} \\
\end{array}} \right. ,
\end{equation}
for a constant margin $\kappa_{\beta} > 1$. Consequently, we also define
\begin{equation}\label{eq:relaxed_cI}
    c_{\beta,\mathcal{I}}(p) = [g_{\beta,\mathcal{I}}(p_i)]^K_{i=1} .
\end{equation}
Clearly the path $p$ is contained in the relaxed feasible set defined by $g_{\beta,\mathcal{I}}$. More formally, this means that $c_{\beta,\mathcal{I}}(p) < 0$.
}

\edit{
If $\beta <0$ for the straight line path, then no homotopy is needed at all: the straight line is feasible and optimal. If $\beta > 0$ we exploit the nature of the interior point solver to drive the path towards feasibility. If we choose $\kappa_{\beta}$ close to (but still greater than) $1$, then the boundary of each violated constraint's relaxation will be very close to some corner of $p$, and the interior path iteration will naturally push the path towards the interior of the (relaxed) feasible region. By using a large barrier parameter $\mu_{\rm hi}$, we can obtain a new path that will not be close to any boundary of the relaxed constraint vector $g_{\beta,\mathcal{I}}$, allowing us to reduce the relaxation parameter $\beta$. Thus, we just need to recompute $\beta$ and repeat this process until $\beta$ is close enough to $0$, indicating that the corner points of the path satisfy the original (non-relaxed) constraints. If this process stagnates for any reason ($\beta$ stops decreasing), we report failure under suspicion that a feasible path may not exist (see Fig.~\ref{fig:case9mod}).
}

\edit{
Pseudo-code of the complete shortest path algorithm, including the generation of a feasible path, is given by Algorithm~1. We reuse the variables of each relaxation step as a warm start for the next relaxation to reduce computation time. Upon finding a feasible path, we compute the \emph{shortest} path by calling the interior point solver with a small barrier parameter $\mu$. To determine if the algorithm is making enough progress in decreasing $\beta$, we consider its relative decrease, $\delta_{\beta}$. If at any iteration $\delta_{\beta}$ is not greater than the user-specified tolerance $\epsilon_{\rm tol}$, the algorithm assumes that $\beta$ has stagnated and reports failure. Conversely, if $\delta_{\beta} > \epsilon_{\rm tol}$ we assume that enough progress has been made, and we compute new relaxation steps. In particular, this means that the interior point iteration does not need to run until full convergence during the homotopy process; the execution can be interrupted as soon as the new $\beta$ has decreased enough.
}

\edit{
Some OPF cases have inequalities that are so close that they roughly behave like equalities, making the feasible region nearly a lower-dimensional manifold with no interior. In such cases, the interior point algorithm may present convergence difficulties or even fail completely. As a safeguard against these issues, the last solver call uses the relaxed constraints $g_{\beta,\mathcal{I}}$ with a small relaxation parameter $\beta = \epsilon_{\rm comp}$. This slightly increases the size of the feasible region's interior, so that the solver has enough ``space'' in the feasible set to move the candidate path towards the solution. As a consequence, if $0 \leq \beta < \epsilon_{\rm comp}$ for the straight line, then the algorithm still treats it as feasible and accepts the path.
}

\begin{algorithm}
 \caption{Shortest Path Algorithm (Outer Loop)}
 \edit{
 \begin{algorithmic}[1]
 \Procedure{ShortestPath}{$f$, $g_{\mathcal{I}}$, $\{t_k\}_{k=0}^{K+1}$, $\kappa_{\beta}$, $\mu_{\rm hi}$, $\mu_{\rm lo}$, $\epsilon_{\rm tol}$, \hspace*{0.8em}${\rm iter}_{\max}$, $\tau$, $\gamma$, $\eta$, $\nu_0$, $\kappa_{\nu}$, $\epsilon_{\rm comp}$, $\rho_{\max}$}
  \State compute $u_k$ from \eqref{eq:line_path_u} \textbf{and} compute $x_k$ from \eqref{eq:line_path_x}
  \State compute $p$ from \eqref{eq:p} \textbf{and} compute $\beta$ from \eqref{eq:relaxation_vector}
  \State \textbf{if} $\beta < \epsilon_{\rm comp}$ \textbf{then} \textbf{return} $p$, $\beta$
  \State compute $g_{\beta, \mathcal{I}}$ from \eqref{eq:relaxed_constraints} \textbf{and} compute $c_{\beta,\mathcal{I}}(p)$ from \eqref{eq:relaxed_cI}
  \LComment{default values for barrier problem vars:}
  \State $v \gets 0$, $y \gets 0$, $s \gets -c_{\beta, \mathcal{I}}(p)$, $z \gets \mu \vec{1} \oslash s$
  \While{$\beta \geq \epsilon_{\rm comp}$}
  \State $p,s,v,y,z \gets$ call \Call{BarrierSolve}{$f$, $g_{\beta, \mathcal{I}}$, $p$, $\mu_{\rm hi}$, $\ldots$}, but interrupt execution as soon as $\max_{i,j} \left( g_j(p_i) \right) < (1 - \epsilon_{\rm tol}) \beta$
  \State assign $\beta^- \gets \beta$ \textbf{and} compute $\beta$ from \eqref{eq:relaxation_vector}
  \State $\delta_{\beta} = (\beta^- - \beta) / \beta^-$
  \If{$\delta_{\beta} \leq \epsilon_{\rm tol}$ \textbf{and} $\beta \geq \epsilon_{\rm comp}$}
  \State report failure \textbf{and break}
  \EndIf
  \State compute $g_{\beta, \mathcal{I}}$ from \eqref{eq:relaxed_constraints}
  \EndWhile
  \If{$\beta < \epsilon_{\rm comp}$}
  \State assign $\beta \gets \epsilon_{\rm comp}$ \textbf{and} compute $g_{\beta, \mathcal{I}}$ from \eqref{eq:relaxed_constraints}
  \State $p,s,v,y,z \gets$ \Call{BarrierSolve}{$f$, $g_{\beta, \mathcal{I}}$, $p$, $\mu_{\rm lo}$, $\ldots$}
  \EndIf
  \State \textbf{return} $p$, $\beta$
  \EndProcedure
 \end{algorithmic}} 
\end{algorithm}

\section{Numerical Experiments}
This section describes experiments performed to assess the performance of the proposed algorithm. We provide a public implementation of the algorithm, illustrative examples, and experiments on power systems of different scales.

\subsection{Implementation}
We developed a Julia code that implements the shortest path algorithm. The code is publicly available at the following page:

\vspace{0.3em}

\noindent
\href{https://github.com/djturizo/Shortest-Path-OPF}{\textcolor{blue}{\texttt{github.com/djturizo/Shortest-Path-OPF}}}

\vspace{0.3em}

\noindent
All experiments were run using Julia 1.10 on a Windows 11 PC with 32GB of RAM and an AMD Ryzen\textsuperscript{\tiny\texttrademark} PRO 7840U CPU \edit{with 8 physical cores and 16 parallel threads}. Unless specified otherwise, we used the following parameters:
\begin{align*}
K &= 9, \quad &t_k &= 0.05 \cdot k, \quad  &k &= 0,\ldots,K+1, \\
\kappa_{\beta} &= 1.01, \quad &\mu_{\rm hi} &= 0.05, \quad
&\mu_{\rm lo} &= 10^{-5}, \\ 
\epsilon_{\rm tol} &= 10^{-3}, \quad &{\rm iter}_{\max} &= 100, &\tau &= 0.99, \\
\gamma &= 0.5, \quad &\eta &= 10^{-4}, &\nu_0 &=10^{-6}, \\
\kappa_{\nu} &= 0.1, &\epsilon_{\rm comp} &= 10^{-6}, &\epsilon_{\rm ls} &= 10^{-2}, \\
\rho_{\max} &= 100.0.
\end{align*}
The power flow equations were solved using the Newton-Raphson method with a tolerance of $10^{-8}$ and a limit of $20$ iterations (see Step 2 of Algorithm 1). The shortest path algorithm uses a network model with one generator per node at most and rectangular coordinates for the voltage phasors, in order to have quadratic power flow equations and constraints (except for line flow constraints). Some test cases have multiple generators in a single node, but it is possible to compute a single equivalent generator. Angle difference constraints can be written as quadratic inequalities whenever the corresponding angle limit lies in the interval $(-\pi/2,\pi/2)$ (see \cite{opf_qc_relaxation}), which is often the case in practice. 

During the execution of the experiments, we noticed that evaluating the power flow feasibility constraint \eqref{eq:power_flow_constraint} took a significant portion of the execution time, but it was never active. This is consistent with the expectation that the boundary of the power flow feasibility constraint is significantly larger than that of all other constraints, so the feasible set ends up being determined by the standard OPF constraints. This means that the power flow feasibility constraint has no effect at all on the results of the shortest path algorithm (and we confirmed this on the experiments). We thus ignored this constraint in our experiments to increase the execution speed of the algorithm.

\subsection{Example: Two Variants of the 9-Bus Case}
To illustrate how the algorithm works in different situations, we next use the 9-bus OPF case of M{\sc atpower} \cite{matpower_manual}. The system has three generators, at nodes 1 to 3, with node 1 being the slack node. The control variables are the voltage magnitudes of the generators ($V_1,V_2,V_3$) and the active power of non-slack generators ($P_{G2}, P_{G3}$). We consider two variants of the 9-bus case obtained by modifying the system parameters. The first one, called variant 1 from now on, is modified to introduce an obstacle in the feasible region. First we set the generator voltage magnitudes to be $1$ p.u. ($V_1 = V_2 = V_3 = 1$). The control vector in the subspace is chosen as $u = [P_{G2}, P_{G3}]^T$. We generate the obstacle by setting the lower reactive power limit of the generator at bus 3 to $-2$ MVA ($Q_{G3 \min} = -0.02$). For the endpoints we choose $u_0 = [0.5, 0.5]^T$ and $u_{\infty} = [1.5, 1.3]^T$.

We executed the shortest path algorithm, obtaining the results illustrated in Fig. \ref{fig:case9hole}. The feasible region is colored in green, and the relaxations generated by the algorithm are colored in red hues. Later iterations have smaller constraint violations, which lead to tighter relaxations, represented with darker shades of red. The shortest path is computed for each relaxation. Paths corresponding to tighter relaxations are colored with lighter shades of blue for contrast. The figures shows the continuous  deformation of the path as it moves away from the boundary. After multiple iterations of this process, the algorithm obtains a feasible path, and then the final iteration tightens the candidate path while preserving feasibility.

We next consider another modification, called variant 2 from now on, where no feasible path exists. For this variant we used the 9-bus OPF case of {\sc Matpower} \cite{matpower_manual}, modified as in \cite{case9mod}. We also fix the generator voltage magnitudes to the following p.u. values: $V_1 = 0.920, V_2 = 0.935, V_3 = 0.943$. The control vector in the subspace is chosen as $u = [P_{G3}, P_{G2}]^T$. The endpoints are chosen to be from different connected regions. Namely, we chose $u_0 = [0.12, 0.16]^T$ and $u_{\infty} = [1.57, 0.24]^T$.

We executed the shortest path algorithm, obtaining the results illustrated in Fig. \ref{fig:case9mod}. The figure shows how tighter relaxations become narrower around the center in an attempt to eventually break into two components. As a result, the candidate path ends up ``choked'' in this narrow passage, which attempts stretch the path, separating the corner points into two distant clusters. Such a deformation would violate the constant speed constraints that require the corner points to preserve the relative distance between them. As a result, the algorithm is unable to reduce the constraint violations any further, and it appropriately reports failure to find a feasible path.

As a last experiment for this case, we modified the value of $\mu_{\rm lo}$ to observe its effect on the computation of the shortest path from a given feasible path (Step 17 of Algorithm 1). For this experiment, we consider the variant 1 of the 9-bus case and we solve the shortest path problem for multiple values of $\mu_{\rm lo}$ in the range $[10^{-12}, 10^{-2}]$. For each value of $\mu_{\rm lo}$, we compute the length of the shortest path found as the percentage increase over the path length of the unconstrained solution (i.e., the straight line joining the endpoints). As shown in Fig.~\ref{fig:barrier_parameter}, the feasible path generated by the homotopy process may be significantly larger than the shortest path, warranting the last optimization process that is executed with a lower barrier parameter. For small values of $\mu_{\rm lo}$, observe that the solution does not change significantly.

\newcommand{\figscale}{0.5}
\begin{figure}[tbp!]
    \centering
  \includegraphics[scale=\figscale]{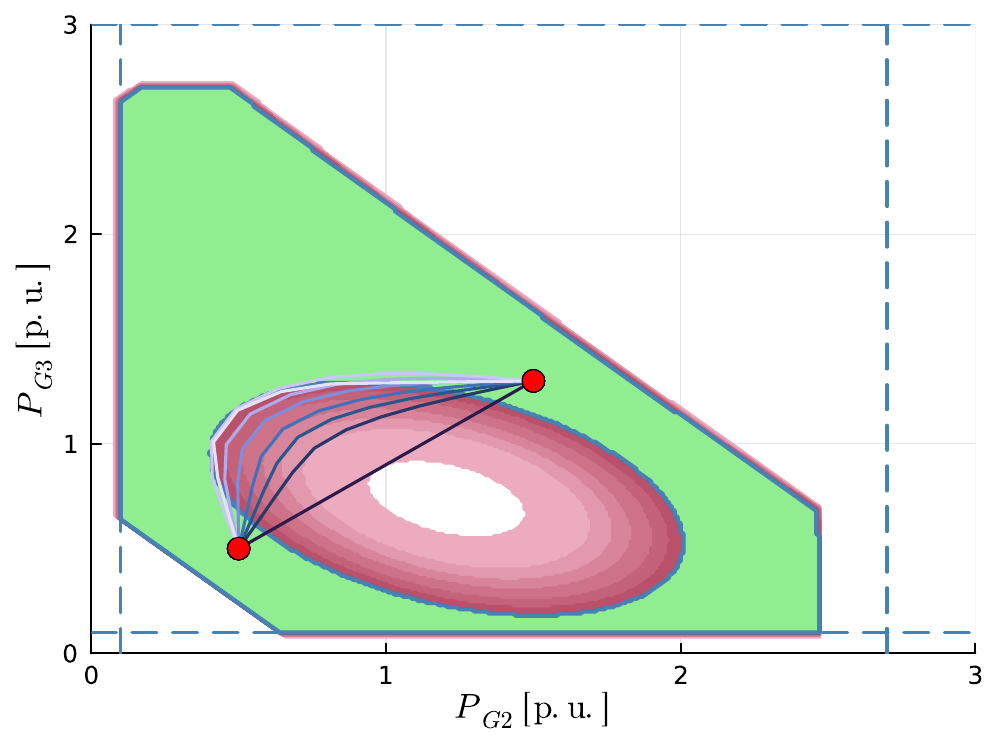}
  	\caption{Variant 1 of the 9-bus case. The straight line path is not feasible, but the algorithm deforms the path to achieve feasibility.}
  	\label{fig:case9hole}
\end{figure}

\begin{figure}[tbp!]
    \centering
  \includegraphics[scale=\figscale]{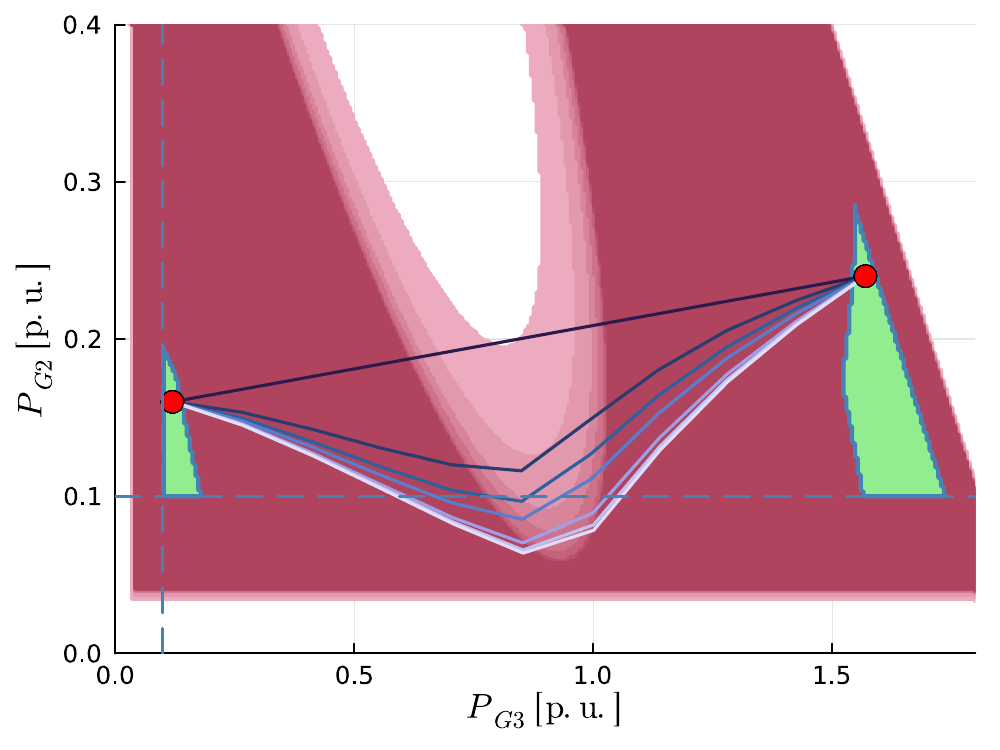}
  	\caption{Variant 2 of the 9-bus case. The endpoints are disconnected, so the algorithm fails to find a feasible path.}
  	\label{fig:case9mod}
   \vspace{-0.5em}
\end{figure}

\begin{figure}[tbp!]
    \centering
  \includegraphics[scale=\figscale]{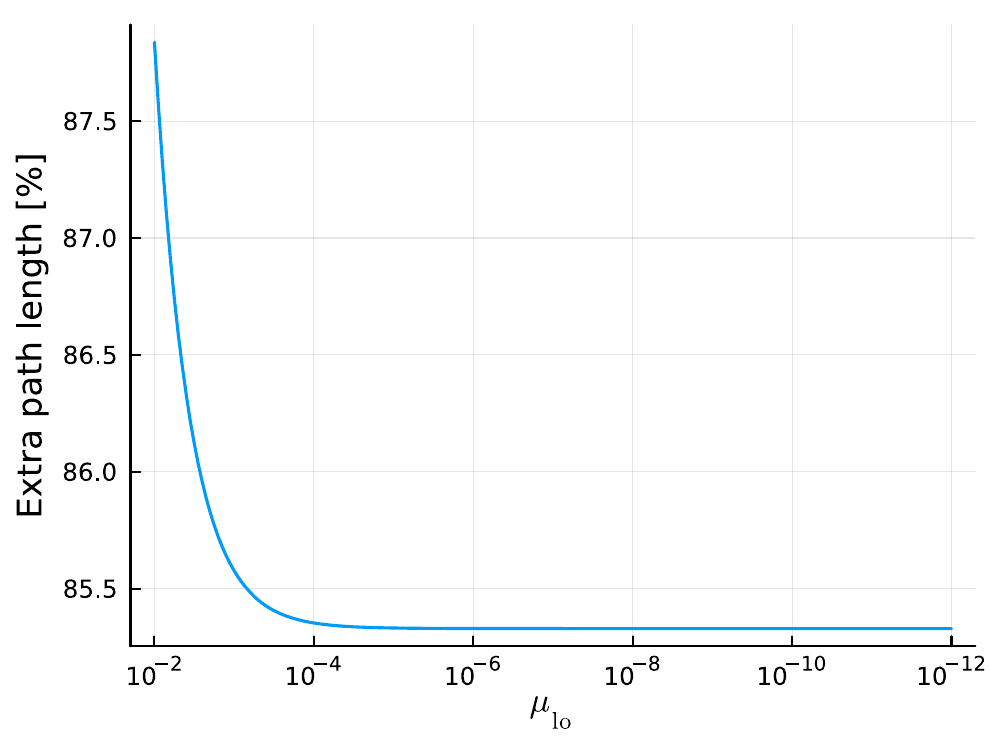}
  	\caption{Variant 1 of the 9-bus case. For smaller barrier parameters, the path length decreases until stabilizing at the shortest path. However, very small barrier parameters introduce numerical artifacts.}
  	\label{fig:barrier_parameter}
   \vspace{-1em}
\end{figure}

\subsection{Multiple scale OPF cases}
For this experiment, we used multiple OPF benchmark test cases from the Power Grid Library PGLib \cite{pglib}. We selected nine cases of different sizes, ranging from 14 to 118 buses. Since these cases have high-dimensional feasible spaces that are hard to visualize, selecting non-trivial endpoints (where the straight line is not feasible) is not always straightforward. We therefore follow the heuristic presented in \cite{lee_feasible_path}: namely, we selected the endpoints as the solution of the minimum loss problem and the OPF solution. For each test case, we computed the maximum constraint violation (relaxed with parameter $\epsilon_{\rm comp}$) over the path points in the starting straight line (before running the algorithm, \edit{in the column called ``Max. con. before''}) and over the final path resulting from running the algorithm \edit{(in the column called ``Max. con. after'')}, ignoring the endpoints (because they are fixed and not modified by the algorithm). If the maximum constraint violation after running the algorithm is negative, then the final path found is feasible and the algorithm has thus identified a shortest path (in a local sense, at least). 

We also computed solution and objective function metrics as follows: let $p(t)$ be the piece-wise linear shortest path approximation (as defined in \eqref{eq:p_t}) resulting from the algorithm, and let $L(t)$ be the straight line path associated with the endpoints. With both $p$ and $L$ parameterized by arclength, we computed the relative difference between the paths as
\begin{equation*}
    \textrm{path-diff\%} = \frac{\int_0^1{{\|p(t) - L(t)\| dt}}}{\int_0^1{{\|L(t)\| dt}}} \times 100\% .
\end{equation*}
Similarly, we computed the relative objective function increase, or gap, with respect to the value at the straight line:
\begin{equation*}
    \textrm{obj-fun-gap\%} = \frac{\int_0^1{{\|p(t)\| dt}} - \int_0^1{{\|L(t)\| dt}}}{\int_0^1{{\|L(t)\| dt}}} \times 100\% .
\end{equation*}

The results are reported in Table I. The algorithm succeeded in finding a locally shortest path on all test cases \edit{In particular, our method found the shortest path for the 57-bus, 89-bus, 162-bus, 200-bus, 240-bus and 300-bus cases, where the approach of \cite{lee_feasible_path} failed to generate a feasible path (in the sense that this approach could not drive the optimality gap with respect to the OPF solution endpoint to a margin below 1\%, or diverged entirely). We also remark that the paths found by our method are composed of 10 linear pieces, regardless of the system size. This is by design, as we chose $K=9$ for the experiments. In contrast, the feasible paths generated by \cite{opf_sequence_path} have linear pieces equal to the number of controlled variables, which means the amount of control actions increases with the system size. For example, for case 300 the approach of \cite{opf_sequence_path} generates a feasible path with 189 linear pieces, whereas our approach generates a feasible path 10 linear pieces, with the possibility of producing paths with more or less linear pieces if desired.}

In many test cases the straight line is slightly infeasible, \edit{and as a result only small deformations are required to obtain a feasible path.} One notable exception being the 60-bus case where the straight line has violations as large as is $2.22$ p.u., \edit{but even in that case the difference between the shortest path and the straight line path is around {\ttilde}$1\%$.} In the 14- and 30-bus cases, the straight line is feasible, so the algorithm immediately accepts the straight line without performing the homotopy process or the final optimization step. \edit{These results suggest that the OPF feasible regions of practical cases are often ``almost'' convex, in the sense that if a straight line joining two feasible points is not feasible, usually a small modification of the path is all it takes to recover feasibility. We note that Variant 1 of case 9 is not a typical test case, as it has been modified to introduce a large obstacle between the endpoints. Consequently, the shortest path for that case is much larger than the straight line path.}

We also executed the algorithm on eight test cases selected from \cite{molzahn_cases}, which were crafted specifically to be challenging for OPF solvers. The results for this second batch of test cases are shown in Table II. As expected, these test cases proved to be more challenging, as in one of the eight presented cases the algorithm failed to find a feasible path. We remark that it is possible that the endpoints of those cases are not connected, but the algorithm may also fail even if a feasible path exists (after all, this is a non-convex optimization problem). For the seven remaining cases the straight line was already feasible in two of them, and for the other five the algorithm succeeded in generating a locally shortest path. We observed that the relative path differences are usually larger than in the PGLib test cases, and we suspect this tendency is due to more pronounced non-convexities resulting from the fact that these test cases have been engineered to challenge OPF solvers.

\subsection{Tests on the number of control actions}
\edit{As we mentioned in the introduction, each linear segment of the path represents a single control action, and so the number of linear segments is equal to the number of control actions. While it is desirable for the operator to use a path with few control actions, this increases the risk of violations during the transition. Thus, there is a trade-off between simplicity and feasibility that must be considered when choosing the number of control actions. To study this phenomenon, we performed an experiment where we executed the shortest path algorithm on some text cases while varying the number linear segments ($K+1$) geometrically from 2 to 128. The breakpoints were spaced uniformly for all cases. The results of this experiment are shown in Table \ref{tbl:breakpoint_tests}.}

\edit{From the results, we see that the outcome of whether the algorithm finds a feasible path or not remains consistent for each test case, regardless of the number of segments. This suggests that our method could be reliably used as an oracle for the likelihood of the existence of a feasible path between the endpoints. Another observation is that the path length increases with $K$ for all test cases. This is to be expected: as we increase the number of segments, the piecewise linear path becomes a better approximation of the continuous shortest path. Table \ref{tbl:breakpoint_tests} also reports the number of iterations required for the algorithm to reach a decision. In particular, apart from Variant 1 of the 9-bus case, there is a trend where increasing $K$ leads to an increasing on the number of iterations. This means that, while the cost per iteration scales linearly with $K$, the total execution cost will scale superlinearly. This result also suggests that the optimization problem becomes more challenging as $K$ increases. While we used uniform breakpoint spacing for this experiment, it is possible to implement better informed methodologies for breakpoint spacing. An adaptive spacing technique could possibly allocate more breakpoints to sections of the feasible path close to strong non-convexities of the boundary, while using less breakpoints in section of the path far from the boundary. Development of such techniques is left for future work.}

\begin{table*}[t]
\centering
\caption{Results of running the shortest path algorithm on PGLib test cases}
\label{tbl:test_cases}
\hspace*{-1em}
\texttt{\edit{
\begin{tabular}{|l|r|r|r|l|l|r|r|r|}
\hline
\textrm{Test case} & $ n $ & $g$ & \textrm{Max. con.} & \textrm{Exec.} & \textrm{Found} & \textrm{Max. con.} & \textrm{Path} & \textrm{Obj. fun.}\\
 & & & \textrm{before [p.u.]} & \textrm{time [s]} & \textrm{path?} & \textrm{after [p.u.]} & \textrm{diff.} & \textrm{gap} \\
\hline
\textrm{case9 (Variant 1)}  &9      &2      &2.79E-2        &0.2    &Yes    &-6.26E-7       &85.3\% &34.4\%  \\
\hline
\textrm{case14\_ieee}   &14     &5      &-9.94E-7       &0.0    &Yes    &-9.94E-7       &0.00\% &0.00\%  \\
\hline
\textrm{case24\_ieee\_rts}      &24     &11     &9.92E-4        &4.7    &Yes    &-1.00E-6       &2.13\% &0.03\%  \\
\hline
\textrm{case30\_ieee}   &30     &6      &-9.92E-7       &0.0    &Yes    &-9.92E-7       &0.00\% &0.00\%  \\
\hline
\textrm{case39\_epri}   &39     &10     &9.66E-2        &0.5    &Yes    &-2.97E-3       &2.86\% &0.06\%  \\
\hline
\textrm{case57\_ieee}   &57     &7      &2.50E-3        &0.5    &Yes    &-9.99E-7       &1.53\% &0.02\%  \\
\hline
\textrm{case60\_c}      &60     &23     &2.22E+0        &3.6    &Yes    &-1.00E-6       &2.98\% &0.06\%  \\
\hline
\textrm{case73\_ieee\_rts}      &73     &33     &9.59E-4        &1.6    &Yes    &-1.00E-6       &3.62\% &0.10\%  \\
\hline
\textrm{case89\_pegase} &89     &12     &2.27E-2        &2.9    &Yes    &-8.59E-4       &1.52\% &0.02\%  \\
\hline
\textrm{case118\_ieee}  &118    &54     &2.42E-2        &3.4    &Yes    &-1.00E-6       &3.64\% &0.09\%  \\
\hline
\textrm{case162\_ieee\_dtc}     &162    &12     &1.36E-4        &3.6    &Yes    &-3.63E-5       &1.75\% &0.02\%  \\
\hline
\textrm{case200\_activ} &200    &38     &2.20E-2        &5.0    &Yes    &-1.00E-6       &3.82\% &0.10\%  \\
\hline
\textrm{case240\_pserc} &240    &53     &6.00E-1        &168.4  &Yes    &-1.46E-3       &3.02\% &0.06\%  \\
\hline
\textrm{case300\_ieee}  &300    &69     &5.93E-2        &44.0   &Yes    &-1.00E-6       &3.64\% &0.10\%  \\
\hline
\textrm{case500\_goc}   &500    &113    &1.29E-1        &69.7   &Yes    &-1.35E-4       &7.47\% &0.40\%  \\
\hline
\end{tabular}%
}}\\[0.3em]
\edit{
``Max. con. before'': Maximum constraint violation over the path points in the starting straight line (before running the algorithm).\\
``Max. con. after'': Maximum constraint violation for the final path resulting from running the algorithm.
}
\vspace*{-1em}
\end{table*}

\begin{table*}[t]
\centering
\caption{Results of running the shortest path algorithm on the test cases of \cite{molzahn_cases}}
\label{tbl:test_cases_hard}
\hspace*{-1em}
\texttt{\edit{
\begin{tabular}{|l|r|r|r|l|l|r|r|r|}
\hline
\textrm{Test case} & $ n $ & $g$ & \textrm{Max. con.} & \textrm{Exec.} & \textrm{Found} & \textrm{Max. con.} & \textrm{Path} & \textrm{Obj. fun.}\\
 & & & \textrm{before [p.u.]} & \textrm{time [s]} & \textrm{path?} & \textrm{after [p.u.]} & \textrm{diff.} & \textrm{gap} \\
\hline
\textrm{nmwc3acyclic\_connected\_feasible\_space}       &3      &2      &-1.85E-2       &0.0    &Yes    &-1.85E-2       &0.00\% &0.00\%   \\
\hline
\textrm{nmwc3acyclic\_disconnected\_feasible\_space}    &3      &2      &1.96E-5        &4.3    &Yes    &-3.65E-5       &0.72\% &0.01\%   \\
\hline
\textrm{nmwc3cyclic}    &3      &2      &9.28E-3        &0.0    &No     &3.18E-3        &-      &-       \\
\hline
\textrm{nmwc4}  &4      &2      &-2.12E-4       &0.0    &Yes    &-2.12E-4       &0.00\% &0.00\%  \\
\hline
\textrm{nmwc5}  &5      &2      &2.34E-2        &0.0    &Yes    &-1.46E-3       &2.57\% &0.05\%  \\
\hline
\textrm{nmwc14} &14     &5      &9.26E-4        &0.1    &Yes    &-2.31E-5       &4.00\% &0.11\%  \\
\hline
\textrm{nmwc24} &24     &11     &3.71E-3        &0.3    &Yes    &-9.92E-7       &2.41\% &0.04\%  \\
\hline
\textrm{nmwc57} &57     &7      &2.90E-3        &0.7    &Yes    &-3.75E-5       &5.76\% &0.23\%  \\
\hline
\end{tabular}
}}
\vspace*{-1em}
\end{table*}

\begin{table*}[t!]
\centering
\caption{Results of running the shortest path algorithm with differing numbers of breakpoints for selected test cases}
\label{tbl:breakpoint_tests}
\hspace*{-1em}
\texttt{\edit{
\begin{tabular}{|l|r|r|r|l|l|r|r|r|}
\hline
\textrm{Test case} & $K$ & \textrm{\# iter.} & \textrm{Max. con.} & \textrm{Exec.} & \textrm{Found} & \textrm{Max. con.} & \textrm{Path} & \textrm{Obj. fun.}\\
 & & & \textrm{before [p.u.]} & \textrm{time [s]} & \textrm{path?} & \textrm{after [p.u.]} & \textrm{diff.} & \textrm{gap} \\
\hline
\multirow{6}{*}{\textrm{case9 (Variant 1)}}  &1      &54     &2.79E-2        &0.6    &Yes    &-3.91E-7       &73.6\% &24.2\%  \\
\cline{2-9}
        &3      &48     &2.79E-2        &0.5    &Yes    &-4.64E-7       &80.7\% &31.6\%  \\
\cline{2-9}
        &7      &44     &2.79E-2        &0.6    &Yes    &-5.15E-7       &84.9\% &34.2\%  \\
\cline{2-9}
        &15     &24     &2.79E-2        &0.9    &Yes    &-9.46E-7       &85.7\% &34.7\%  \\
\cline{2-9}
        &31     &15     &2.79E-2        &2.1    &Yes    &-1.82E-6       &85.9\% &34.8\%  \\
\cline{2-9}
        &63     &12     &2.79E-2        &3.1    &Yes    &-1.08E-6       &86.0\% &34.9\%  \\
\cline{2-9}
        &127    &13     &2.79E-2        &8.1    &Yes    &-7.22E-6       &86.0\% &34.9\%  \\
\hline
\multirow{6}{*}{\textrm{case57\_ieee}}   &1      &7      &2.50E-3        &0.1    &Yes    &-9.64E-7       &1.11\% &0.01\%  \\
\cline{2-9}
        &3      &8      &2.50E-3        &0.2    &Yes    &-9.50E-7       &1.60\% &0.02\%  \\
\cline{2-9}
        &7      &8      &2.50E-3        &0.3    &Yes    &-9.99E-7       &1.35\% &0.01\%  \\
\cline{2-9}
        &15     &9      &2.50E-3        &0.8    &Yes    &-1.00E-6       &1.17\% &0.01\%  \\
\cline{2-9}
        &31     &12     &2.50E-3        &2.0    &Yes    &-1.00E-6       &1.30\% &0.01\%  \\
\cline{2-9}
        &63     &13     &2.50E-3        &5.0    &Yes    &-1.00E-6       &1.77\% &0.02\%  \\
\cline{2-9}
        &127    &14     &2.50E-3        &13.4   &Yes    &-1.00E-6       &2.47\% &0.05\%  \\
\hline
\multirow{6}{*}{\textrm{nmwc3cyclic}}    &1      &8      &9.28E-3        &0.0    &No     &3.18E-3        &-      &-       \\
\cline{2-9}
        &3      &8      &9.28E-3        &0.0    &No     &3.18E-3        &-      &-       \\
\cline{2-9}
        &7      &8      &9.28E-3        &0.0    &No     &3.18E-3        &-      &-       \\
\cline{2-9}
        &15     &8      &9.28E-3        &0.0    &No     &3.18E-3        &-      &-       \\
\cline{2-9}
        &31     &9      &9.28E-3        &0.1    &No     &3.18E-3        &-      &-       \\
\cline{2-9}
        &63     &9      &9.29E-3        &0.2    &No     &3.19E-3        &-      &-       \\
\cline{2-9}
        &127    &106    &9.29E-3        &8.7    &No     &3.19E-3        &-      &-       \\
\hline
\multirow{6}{*}{\textrm{nmwc57}} &1      &9      &2.90E-3        &0.1    &Yes    &-1.02E-5       &5.31\% &0.14\%  \\
\cline{2-9}
        &3      &11     &2.90E-3        &0.2    &Yes    &-1.45E-5       &5.40\% &0.19\%  \\
\cline{2-9}
        &7      &14     &2.90E-3        &0.5    &Yes    &-3.25E-5       &5.59\% &0.21\%  \\
\cline{2-9}
        &15     &15     &2.90E-3        &1.0    &Yes    &-2.80E-5       &5.91\% &0.25\%  \\
\cline{2-9}
        &31     &21     &2.90E-3        &2.9    &Yes    &-1.76E-5       &6.49\% &0.30\%  \\
\cline{2-9}
        &63     &35     &2.90E-3        &10.7   &Yes    &-8.60E-6       &7.66\% &0.43\%  \\
\cline{2-9}
        &127    &31     &2.90E-3        &23.8   &Yes    &-4.55E-6       &9.08\% &0.60\%  \\
\hline
\end{tabular}
}}
\vspace*{-1em}
\end{table*}

\section{Conclusions}
\edit{In this paper, we developed an algorithm for computing a discretized feasible path from an initial feasible operating point to an optimal one (or between any two feasible points in general), such that the amplitude of the control actions required to transition from one point to another is minimized. Minimization of control action amplitude is equivalent to minimizing the transition path length, which leads to a discretized shortest path optimization problem. The path is represented as a sequence of intermediate feasible points, the number and relative spacing of which can be specified a priori.} The algorithm computes the intermediate points by solving a nonlinear optimization problem via a specialized interior point method, provided an initial feasible path is given. By leveraging the nature of barrier functions in interior point methods, an initial feasible path is found by solving a sequence of relaxed, but increasingly tighter relaxations of the shortest path problem, where in the initial relaxation the straight line joining the endpoints is feasible. The resulting sequence of shortest paths converges to a feasible path of the original problem in a finite number of iterations. The interior point solver for the algorithm was modified to exploit the sparse block tridiagonal structure of the shortest path problem. Multiple numerical experiments show that the proposed algorithm is can effectively compute a shortest path for a specified number of intermediate points.

The algorithm we developed tackles the issues of the number and amplitude of control actions in the problem of transitioning between operating points. One issue not considered in this work is feasibility of the path in the continuous sense. While our algorithm provides a sequence of intermediate points that are guaranteed to be feasible, the lines joining them may cross the boundary of the feasible set. \edit{Possible avenues of future work are extending the current algorithm with a methodology to provide mathematical guarantees that the line pieces comprising the discrete path are entirely contained in the feasible set and the development of adaptive, non-uniform breakpoint spacing strategies.}

\bibliographystyle{IEEEtran}
\bibliography{IEEEabrv,refs}

\ifx\arxivversion\undefined
\else
\section*{Appendices}

\subsection*{Appendix A: Structure of the Lagrangian Hessian}

The Lagrangian of the Hessian, $\nabla^2_{pp} L$, has components associated with the objective function and the constraints. Each of these components has a specific matrix structure the we analyze next. First we compute $\nabla^2_{pp} L$ using the independence of $s$ on $p$:
\begin{subequations}
\begin{align}
\nabla^2_{pp} L &= \nabla^2_{pp} \phi + \nabla^2_{pp} (v^T [f(p_i)]^K_{i=1}) + \nabla^2_{pp} (y^T c_{\mathcal{E}}) \nonumber \\
&\quad + \nabla^2_{pp} (z^T ([g_{\mathcal{I}}(p_i)]^K_{i=1} + s)), \\
\nabla^2_{pp} L &= \nabla^2_{pp} \phi + \sum_{i=1}^{K}\sum_{j =1}^{2n}{\left(v_i\right)_j \nabla^2_{pp} f_j(p_i)} + \sum_{j \in \mathcal{E}} {(y)_j \nabla^2_{pp} c_j} \nonumber \\
&\quad + \sum_{i=1}^{K}\sum_{j \in \mathcal{I}}{\left(z_i\right)_j \nabla^2_{pp} g_j(p_i)}.
\end{align}
\end{subequations}
First we analyze the Hessian of the objective function, $\nabla^2_{pp} \phi$. Let $I\! \in\! \mathbb{R}^{2g \times 2g}$\! be the $2g \times 2g$ identity matrix, and define the matrix $I'$ as
\begin{equation}
\label{eq:Iprime}
    I' = \begin{bmatrix}
        I & 0 \\
        0 & 0_{2n \times 2n}
    \end{bmatrix} \in \mathbb{R}^{2(g+n) \times 2(g+n)}.
\end{equation}
The derivatives of $\phi$ are%
\begin{subequations}
\begin{align}
\frac{\partial^2 \phi}{\partial p_i \partial p^T_i} &= \frac{2}{K+1} (w_i + w_{i+1})I', & i=1,\ldots,K, \\
\frac{\partial^2 \phi}{\partial p_{i-1} \partial p^T_i} &= \frac{\partial^2 \phi}{\partial p_i \partial p_{i-1}} = - \frac{2}{K+1} w_i I', & i=2,\ldots,K, \\
\frac{\partial^2 \phi}{\partial p_i \partial p_j} &= 0, & |i-j| > 1,
\end{align}
\end{subequations}
so $\nabla^2_{pp} \phi$ is a constant, symmetric, and block tridiagonal matrix with symmetric blocks of size $2(g+n) \times 2(g+n)$. We can write $\nabla^2_{pp} \phi$~as%
\begin{subequations}
\label{eq:nabla2_phi}
\begin{align}
\nabla^2_{pp} \phi &= \frac{2}{K+1} (Y \otimes I'), \\
Y &= \left[\begin{matrix}
w_1 + w_2 & -w_2 &  &   \\
-w_2 & \ddots & \ddots &   \\
 &  \ddots & \ddots & -w_K \\
 &   & -w_K & w_K + w_{K+1}
\end{matrix}\right],
\end{align}
\end{subequations}
where $\otimes$ denotes the Kronecker product. The matrix $Y$ can be seen as the admittance matrix of a single loop circuit with line positive resistances given by $w_k,k=2,\ldots,K$, and two shunts corresponding to $w_1$ and $w_{K+1}$. This admittance matrix is guaranteed to be invertible \cite{turizo}. In particular, $Y$ can be factored as in \cite{turizo} to show that $Y$ is positive definite. Moreover, the eigenvalues of the Kronecker product correspond to all pairwise products between eigenvalues of the two factors (see exercise 7.8.11~(b) of \cite{eigenvalues}), so $\nabla^2_{pp} \phi$ is positive semidefinite. 

Next we consider the equality term $\nabla^2_{pp} (y^T c_{\mathcal{E}})$. We compute the derivative terms of $c_j, j \in \mathcal{E}$ as
\begin{subequations}
\begin{align}
\frac{\partial^2 c_j}{\partial p_j \partial p^T_j} &= 2(w_j - w_{j+1})I', \\
\frac{\partial^2 c_j}{\partial p_{j-1} \partial p^T_{j-1}} &= 2w_j I', \\
\frac{\partial^2 c_j}{\partial p_{j+1} \partial p^T_{j+1}} &= -2w_{j+1} I', \\
\frac{\partial^2 c_j}{\partial p_j \partial p^T_{j-1}} &= \frac{\partial^2 c_j}{\partial p_{j-1} \partial p^T_j} = -2w_j I', \\
\frac{\partial^2 c_j}{\partial p_j \partial p^T_{j+1}} &= \frac{\partial^2 c_j}{\partial p_{j+1} \partial p^T_j} = 2w_{j+1} I', \\
\frac{\partial^2 c_j}{\partial p_i \partial p_k} &= 0, \qquad \textrm{any other case}.
\end{align}
\end{subequations}
With a slight abuse of notation, we define
\begin{equation}
w_0 = (y)_0 = (y)_{K+1} = 0 \in \mathbb{R}.
\end{equation}
Notice that
\begin{equation}
\frac{\partial^2 (y^T c_{\mathcal{E}})}{\partial p_i \partial p^T_j} = \sum_{j \in \mathcal{E}} {(y)_j \frac{\partial^2 c_j}{\partial p_i \partial p^T_j}}.
\end{equation}
Therefore, we can write, for $j=1,\ldots,K$, that
\begin{subequations}
\label{eq:hessian_ytce}
\small
\begin{align}
\frac{\partial^2 (y^T c_{\mathcal{E}})}{\partial p_{j-1} \partial p^T_j} &= \frac{\partial^2 (y^T c_{\mathcal{E}})}{\partial p_j \partial p_{j-1}} = -2w_j ((y)_j - (y)_{j-1}) I', \\
\frac{\partial^2 (y^T c_{\mathcal{E}})}{\partial p_j \partial p^T_j} &= 2[w_j ((y)_j - (y)_{j-1}) + w_{j+1} ((y)_{j+1} - (y)_j)]I', \\
\frac{\partial^2 (y^T c_{\mathcal{E}})}{\partial p_i \partial p_j} &= 0, \qquad |i-j| > 1,
\end{align}
\end{subequations}
so the matrix $\nabla^2_{pp} (y^T c_{\mathcal{E}})$ is symmetric and block tridiagonal (with symmetric blocks of size $2(g+n) \times 2(g+n)$).

Next we consider the inequality term of the Lagrangian Hessian, $\nabla^2_{pp} (z^T ([g_{\mathcal{I}}(p_i)]^K_{i=1} + s))$. As every inequality constraint depends only on one specific $p_i$, it is easy to see that the inequality term is block diagonal, with symmetric blocks of size $2(g+n) \times 2(g+n)$. The same argument applies to the power flow term $\nabla^2_{pp} (v^T [f(p_i)]^K_{i=1})$, so it must be block diagonal with symmetric blocks of size $2(g+n) \times 2(g+n)$.
This implies that the Lagrangian Hessian, $\nabla^2_{pp} L$, is symmetric and block tridiagonal (with symmetric blocks of size $2(g+n) \times 2(g+n)$). In particular, we also have that $\nabla^2_{pp} (L - \phi - y^T c_\mathcal{E})$ is block diagonal, with symmetric blocks of size $2(g+n) \times 2(g+n)$.

\subsection*{Appendix B: Implementation Details of the Newton Iteration}

\par\indent
\subsubsection{Indefiniteness Correction}\par\indent

The computation of the Newton step is performed by solving \eqref{eq:newton_step_full}. We also require $\Delta p$ to be a descent direction of the objective function, in order to ensure sufficient progress towards the optimal solution across iterations. A sufficient condition for $\Delta p$ being a descent direction corresponds to the inertia of the Newton matrix of \eqref{eq:newton_step_full} being equal to $\left(K(2g+2n + |\mathcal{I}|), K(2n + |\mathcal{I}| + 1), 0\right)$ \cite{inertia_condition}. This condition is guaranteed in turn if the Hessian $\nabla_{pp}^2 L$ is positive definite (see \cite{nocedal}). The standard way to enforce this condition (as is done in IPOPT, for example \cite{ipopt}) is to factor the Newton matrix using the Bunch-Kaufman algorithm (see \cite{bunchkaufman,bounded_bk}) and then verify the inertia condition. If the inertia is not correct, then a positive diagonal perturbation is added to $\nabla^2_{pp}L$ and the Newton matrix is refactored. The process is repeated using increasingly larger perturbations until the inertia condition is satisfied \cite{ipopt}. This approach works for generic problems, but may require multiple factorizations of the matrix. Instead, we adopted a simpler heuristic derived from the structure of our specific problem to guarantee a descent direction, at the cost of additional line search evaluations (see the next subsection for details on the line search procedure).

Recall that $\nabla^2_{pp}\phi \succeq 0$, so any source of indefiniteness must come from $\nabla^2_{pp}(L - \phi)$, which is block tridiagonal:
\begin{subequations}
\begin{sizeddisplay}{\small}{\ignorespaces}
\begin{align}
&\nabla^2_{pp}(L - \phi) = \nabla^2_{pp} (v^T [f(p_i)]^K_{i=1} + z^T ([g_{\mathcal{I}}(p_i)]^K_{i=1} + s)) \nonumber \\
&\quad + \nabla^2_{pp} (y^T c_{\mathcal{E}}), \\
&\nabla^2_{pp} (v^T [f(p_i)]^K_{i=1} + z^T ([g_{\mathcal{I}}(p_i)]^K_{i=1} + s)) = \nonumber \\
&+ \left[\begin{matrix}
\nabla^2_{p_1 p_1}(v^T_1 f(p_1) + z^T_1 g_{\mathcal{I}}(p_1)) & & \\
 & \ddots & \\
 & & \nabla^2_{p_K p_K}(v^T_K f(p_K) + z^T_K g_{\mathcal{I}}(p_K))
\end{matrix}\right].
\end{align}
\end{sizeddisplay}
\end{subequations}
We want to compute a perturbation matrix $S$, hopefully small, such that $\nabla^2_{pp}L + S \succ 0$.  To this end we consider a block diagonal matrix $S = \Ssub{\mathcal{E}} + \Ssub{\mathcal{I}} + \delta_S I$ for a small $\delta_S > 0$ such that $\nabla^2_{pp} (y^T c_{\mathcal{E}}) + \Ssub{\mathcal{E}} \succeq 0$ and $\nabla^2_{pp} (v^T [f(p_i)]^K_{i=1} + z^T ([g_{\mathcal{I}}(p_i)]^K_{i=1} + s)) + \Ssub{\mathcal{I}} \succeq 0$. First we note that
\begin{equation}
    \nabla^2_{p_i p_j} (y^T c_{\mathcal{E}}) = \begin{bmatrix}
        \nabla^2_{u_i u_j} (y^T c_{\mathcal{E}}) & 0 \\
        0 & 0
    \end{bmatrix},
\end{equation}
for any $i,j=1,\ldots,K$. Thus we can take $\Ssub{\mathcal{E}}$ to be
\begin{equation}\label{eq:S_E}
\Ssub{\mathcal{E}} = l_{\mathcal{E}} \left[ {\begin{matrix}
   I' & 0 & \cdots & 0 \\
   0 & \ddots & \ddots & \vdots \\
   \vdots & \ddots & \ddots & 0 \\
   0 & \cdots & 0 & I' \\
\end{matrix}} \right] \in \mathbb{R}^{2(g+n)K \times 2(g+n)K}
\end{equation}
where $I'$ is the matrix defined in \eqref{eq:Iprime} and $l_{\mathcal{E}} \geq 0$ is an upper bound for the magnitude of the largest negative eigenvalue of $\nabla^2_{pp} (y^T c_{\mathcal{E}})$, i.e., $-l_{\mathcal{E}}$ is a lower bound of the minimum eigenvalue of $\nabla^2_{pp} (y^T c_{\mathcal{E}})$. Next we show that an appropriate $l_\mathcal{E}$ can be obtained with little computational effort.

\begin{theorem}
Let the entries of $\nabla^2_{pp} (y^T c_{\mathcal{E}})$ be given by \eqref{eq:hessian_ytce}, and choose 
\begin{align}
l_{\mathcal{E}} &= -4 \left(1 + \cos \left(\frac{\pi}{K+1}\right) \right) \cdot \min\Big\{0, \nonumber \\
&\hspace*{7em} \min_{j=1,\ldots,K+1}\left[w_j ((y)_j - (y)_{j-1})\right] \Big\}, \label{eq:lE}
\end{align}
then
\begin{equation}
\lambda_{\min}\left(\nabla^2_{pp} (y^T c_{\mathcal{E}})\right) \geq -l_{\mathcal{E}}.
\end{equation}
\end{theorem}
\begin{proof}[Proof]
See Appendix D.
\end{proof}

Now we focus on $\Ssub{\mathcal{I}}$. First we note that
\begin{align}
    &\nabla^2_{p_i p_i}(v^T_1 f(p_1) + z^T_1 g_{\mathcal{I}}(p_1)) = \nonumber \\
    &\quad\quad \begin{bmatrix}
        \nabla^2_{u_i u_i} (z^T_1 g_{\mathcal{I}}(p_1)) & 0 \\
        0 & \nabla^2_{x_i x_i}(v^T_1 f(p_1)) + \nabla^2_{x_i x_i}(z^T_1 g_{\mathcal{I}}(p_1))
    \end{bmatrix},
\end{align}
for any $i,j=1,\ldots,K$ (it is clear from \eqref{eq:power_flow} that the second derivatives of $v^T_1 f(p_1)$ with respect to any two components of $u_i$ are zero). Let $I_i$ be the identity matrix for any $i \in \mathbb{N}$, then we can take $\Ssub{\mathcal{I}}$ as
\begin{equation}\label{eq:S_I}
    \Ssub{\mathcal{I}} = \left[ {\begin{matrix}
   \delta_{u1} I_{2g} & 0 & \cdots & \cdots & 0 \\
   0 & \delta_{x1} I_{2n} & \ddots & & \vdots \\
   \vdots & \ddots & \ddots & \ddots & \vdots \\
   \vdots & & \ddots & \delta_{uK} I_{2g} & 0 \\
   0 & \cdots & \cdots & 0 & \delta_{x1} I_{2n} \\
\end{matrix}} \right],
\end{equation}
where
\begin{subequations}
\begin{align}
\delta_{u1} &= \left\|\nabla^2_{u_i u_i}(z^T_i g_{\mathcal{I}}(p_i))\right\|_F, \\
\delta_{xi} &= \left\|\nabla^2_{x_i x_i}(v^T_i f(p_i)) + \nabla^2_{x_i x_i}(z^T_i g_{\mathcal{I}}(p_i))\right\|_F,
\end{align}
\end{subequations}
for $i = 1, \ldots, K$. Notice that the Frobenius norm is never smaller than the induced 2-norm of a matrix (see exercise 5.6.P23 of \cite{linear_algebra}), so we have that
\begin{subequations}
\begin{align}
&\nabla^2_{p_i p_i}(z^T_i g_{\mathcal{I}}(p_i)) + \delta_{ui} I_{2g} \succeq 0, \\
&\nabla^2_{x_i x_i}(v^T_i f(p_i)) + \nabla^2_{x_i x_i}(z^T_i g_{\mathcal{I}}(p_i)) + \delta_{xi} I_{2n} \succeq 0,
\end{align}
\end{subequations}
for $i=1,\ldots,K$, and thus $\nabla^2_{pp} (v^T [f(p_i)]^K_{i=1} + z^T ([g_{\mathcal{I}}(p_i)]^K_{i=1} + s)) + \Ssub{\mathcal{I}} \succeq 0$. In conclusion, our choices of $\Ssub{\mathcal{E}}$ and $\Ssub{\mathcal{I}}$ guarantee that $\nabla^2_{pp}L + \Ssub{\mathcal{E}} + \Ssub{\mathcal{I}} \succeq 0$, so $\nabla^2_{pp}L + S \succeq \delta_S I \succ 0$, as desired. For the value of $\delta_S$ we choose a small fraction of the 2-norm pf $w_k$, normalized by $K+1$. For example:
\begin{equation}\label{eq:delta_S}
    \delta_S = 10^{-4} \cdot \frac{\|w_k\|_2}{K + 1}.
\end{equation}

\par\indent
\subsubsection{Positivity of Dual Variables and Line Search}\par\indent

The KKT conditions of the interior point formulation require the entries of vectors $z$ and $s$ to have the same sign (see \eqref{eq:KKT_sz}). We require that $s > 0$, as otherwise the barrier terms are undefined. Therefore, we also require that $z > 0$. The Newton step may yield an update $\Delta z$ that makes some of the entries of $z + \Delta z$ non-positive. To prevent this situation, we scale $\Delta z$ using the fraction-to-boundary rule (see \cite{nocedal}) to ensure that the updated vector remains in the positive orthant. The same argument applies for $s$, so a scaling is also computed for $\Delta s$. Another source of difficulties for the Newton iteration is the high nonlinearity that may arise from the interior point formulation. In such cases, the Newton linearization is not a good approximation, except for small step lengths. Thus, if the Newton step does not yield a good enough decrease of the objective function, we should take a smaller step. We achieve this by using a backtracking line search over the Armijo condition of an appropriate merit function \cite{nocedal}. If the current iterate is feasible, then the line search allows us to guarantee that the next iterate is also feasible. In such a case, we can directly update $s$ to be the negative constraint violations a the next iterate. The scaling of $\Delta s$ can still be used for scaling $\Delta p$; this reduces the line search computation time in practice. To summarize, the new iterate variables are computed as follows:
\begin{subequations}\label{eq:line_search_update}
\begin{align}
\alpha^{\max}_z &= \max \left\{\alpha \in [0,1] \,:\, z + \alpha \Delta z \geq (1-\tau) z \right\}, \label{eq:max_step_1_z} \\
\alpha^{\max}_s &= \max \left\{\alpha \in [0,1] \,:\, s + \alpha \Delta s \geq (1-\tau) s \right\}, \label{eq:max_step_1_s} \\
\alpha^{\max}_z &= \min \left\{1, \min \left\{-\tau (z)_i / (\Delta z)_i : (\Delta z)_i < 0 \right\}\right\}, \label{eq:max_step_2_z} \\
\alpha^{\max}_s &= \min \left\{1, \min \left\{-\tau (s)_i / (\Delta s)_i : (\Delta s)_i < 0 \right\}\right\}, \label{eq:max_step_2_s} \\
p^+ &= p + \gamma^M \alpha^{\max}_s \Delta p, \\
y^+ &= y + \gamma^M \alpha^{\max}_z \Delta y, \\
v^+ &= v + \gamma^M \alpha^{\max}_z \Delta v, \\
z^+ &= z + \gamma^M \alpha^{\max}_z \Delta z, \\
s^+ &= - c_{\mathcal{I}}(p^+),
\end{align}
\end{subequations}
for parameters $\tau,\gamma \in (0,1)$. We remark that \eqref{eq:max_step_1_z} and \eqref{eq:max_step_1_s} correspond to the fraction-to-boundary rule (see \cite{nocedal}), and \eqref{eq:max_step_2_z} and \eqref{eq:max_step_2_s} are equivalent definitions that are easier to implement computationally. $M$ is the smallest non-negative integer satisfying the line search conditions. More formally, define the merit function as 
\begin{subequations}
\begin{align}
\psi(\nu, p) &= \psi_o(p) + \nu \psi_c(p), \text{~~where~~}\\
\psi_o(p) &= \phi(p) - \mu \sum_{i=1}^{K}\sum_{j \in \mathcal{I}}{\ln[\max \{-g_j(p_i), 0\}]}, \\
\psi_c(p) &= \left\|c_{\mathcal{E}}(p)\right\|_1 + \left\|f(p)\right\|_1,
\end{align}
\end{subequations}
for some parameter $\nu > 0$, where we use the convention that $\ln 0 = -\infty$. Then the line search conditions are
\begin{subequations}
\label{eq:line_search_conds}
\begin{align}
&\psi(\nu, p + \gamma^M \alpha^{\max}_s \Delta p) \leq \psi(\nu, p) + \eta \Big[(\nabla_p \psi_o(p))^T \gamma^M \alpha^{\max}_s \Delta p \nonumber \\
&\qquad + \nu \left(\psi_c(p + \gamma^M \alpha^{\max}_s \Delta p) - \psi_c(p)\right) \Big], \label{eq:cond_armijo} \\
&\frac{\min_{j=1,\ldots,K+1} \|b_j(p + \gamma^M \alpha^{\max}_s \Delta p)\|_{\infty}}{\max_{j=1,\ldots,K+1} \|b_j(p + \gamma^M \alpha^{\max}_s \Delta p)\|_{\infty}} > \epsilon_{\rm comp}, \label{eq:cond_thm1_1} \\
&\frac{(K+1)^{-1}\left\|\sum^{K+1}_{j=1} q_j(p + \gamma^M \alpha^{\max}_s \Delta p)\right\|_{\infty}}{\max_{j=1,\ldots,K+1} \|q_j(p + \gamma^M \alpha^{\max}_s \Delta p)\|_{\infty}} > \epsilon_{\rm comp}, \label{eq:cond_thm1_2} 
\end{align}
\end{subequations}
for some parameters $\eta, \epsilon_{\rm comp} \in (0,1)$. Equation \eqref{eq:cond_armijo} is the Armijo condition \cite{nocedal}. Equations~\eqref{eq:cond_thm1_1} and \eqref{eq:cond_thm1_2} are the necessary conditions for Theorem \ref{thm:DE_rank}, evaluated at the candidate path $p + \gamma^M \Delta p$. For this paper, we chose $\tau=0.99$, $\gamma=0.5$ and $\eta=10^{-4}$. These values were adapted from typical values used in solvers like IPOPT. 

The parameter $\nu$ deserves special attention, as it controls the trade-off between minimizing the objective function and satisfying the constraints. To ensure feasibility of the solution, $\nu$ must be chosen such that the linear term predicting the change of the constraint violations dominates the linear term predicting the change in the objective function value. As these linear terms change at every iteration, $\nu$ must be adjusted dynamically each time. We do this via a simplified version of the technique proposed in \cite{nocedal_hybrid}, which we describe next. We take $\nu = \nu_0$ for some small constant $\nu_0 > 0$ at the beginning of the interior point method. After computing the Newton step, but before performing the line search, we compute
\begin{equation}
    \nu_{\rm trial} = \frac{(\nabla_p \psi_o(p))^T \Delta p}{(1 - \kappa_{\nu}) \psi_c(p)},
\end{equation}
for some constant $\kappa_{\nu} \in (0,1)$. Then the updated value of $\nu$ to be used on the line search, $\nu^+$, is
\begin{equation}\label{eq:nu_update}
\nu^+ = \left\{ {\begin{array}{ll}
   \nu, & \nu \geq \nu_{\rm trial} \\
   \max\{\nu_{\rm trial}, 2\nu\}, & {\rm else} \\
\end{array}} \right. .
\end{equation}
For this paper we chose $\nu_0 = 10^{-6}$ and $\kappa_{\nu} = 0.1$.

After updating $\nu$, the line search is performed to compute the smallest $M \geq 0$ satisfying \eqref{eq:line_search_conds}. $M$ is computed by trial and error starting with $M=0$ and increasing its value by $1$ until all conditions are satisfied (in the extended real sense) or 
\begin{equation}
\label{eq:cond_sucess}
    \gamma^M \leq \epsilon_{\rm ls},
\end{equation}
for some parameter $\epsilon_{\rm ls} \in (0,1)$. For this paper, we chose $\epsilon_{\rm ls} = 10^{-2}$. It may happen that the Newton direction may not be productive at all, in which case $M \to \infty$. In such cases, the violation of \eqref{eq:cond_sucess} signals the need for a safer step direction, so the current step is discarded, the Hessian indefinitiness correction is applied and the step is recomputed. If \eqref{eq:cond_sucess} is violated again after the indefiniteness correction, then the algorithm reports failure, returns the current solution, and terminates.

\par\indent
\subsubsection{Stopping Criterion}\par\indent

The Newton iteration seeks primal and dual variables that solve the first-order KKT equations, but we are only interested in the values of the primal variables. In some situations, the Newton iteration may converge in the primal variables, but not the dual variables (when the constraint qualifications are ``almost'' not satisfied, for example). To avoid this problem, we follow the approach of IPOPT (see \cite{ipopt}) to define an error metric that is scaled with respect to the dual variables:
\begin{subequations}
\begin{align}
\rho_d &= \max \left\{\rho_{\max}, \frac{\|y\|_1 + \|z\|_1}{|\mathcal{E}| + K|\mathcal{I}|}\right\} / \rho_{\max}, \\
\rho_c &= \max \left\{\rho_{\max}, \frac{\|z\|_1}{K|\mathcal{I}|}\right\} / \rho_{\max}, \\
E_{\mu}(p,s,v,y,z) &= \max\Bigg\{\frac{\|\nabla_p L(p,s,v,y,z)\|_{\infty}}{\rho_d}, \frac{\|s \circ z - \mu \vec{1}\|_{\infty}}{\rho_c}, \nonumber \\
&\quad \|c_{\mathcal{E}}(p)\|_{\infty}, \|f(p)\|_{\infty}\Bigg\}, \label{eq:E_mu}
\end{align}
\end{subequations}
for some parameter $\rho_{\max} > 0$. For this work, we chose $\rho_{\max} = 100$. The stopping criterion for the Newton iteration is
\begin{equation}
E_{\mu}(p,s,v,y,z) \leq \epsilon_{\rm tol},
\end{equation}
where $\epsilon_{\rm tol}$ is the error tolerance specified by the user. This criterion provides the advantage of being robust against problems where the primal variables converge but the dual variables diverge (e.g., the solution does not satisfy the KKT conditions).

\par\indent
\subsubsection{Newton Iteration Algorithm}\par\indent

We have described all the necessary tools to implement a complete Newton iteration for the interior point method. This way we can solve the barrier problem for a fixed barrier parameter $\mu$, as long as we are provided an initial feasible path. Pseudo-code of the procedure given an initial feasible path $p$ is presented in Algorithm 2. The idea is to iteratively compute Newton steps until the error criterion is satisfied (success) or a maximum number of iterations (denoted as ${\rm iter}_{\max}$) is reached (failure). The starting values of the remaining variables can be provided to the algorithm to allow for a warm start, otherwise we choose by default $v=0$, $y = 0$, $s = -c_{\mathcal{I}}(p)$, and $z = \mu \vec{1} \oslash s$. For a small enough $\mu$, the barrier problem's solution will be a good enough approximation to the solution of the original problem. Algorithm 2 requires the user to specify the vectors of power flow equations $f$ and OPF constraints $g_{\mathcal{I}}$, which together completely characterize the power system and the OPF feasible region.

\newcommand{\ICflag}{\rm didCorrection}
\begin{algorithm}
 \caption{Solution of Barrier Problem (Inner Loop)}
 \begin{algorithmic}[1]
 \Procedure{BarrierSolve}{$f$, $g_{\mathcal{I}}$, $p$, $\mu$, $\epsilon_{\rm tol}$, ${\rm iter}_{\max}$, $\tau$, $\gamma$, $\hspace*{0.8em}\eta$, $\nu_0$, $\kappa_{\nu}$, $\epsilon_{\rm comp}$, $\epsilon_{\rm ls}$, $\rho_{\max}$, $v$, $y$, $z$, $s$}
 \LComment{$p$ must be feasible}
 \LComment{Default values of other vars:}
 \LComment{$v \gets 0$, $y \gets 0$, $s \gets -c_{\mathcal{I}}(p)$, $z \gets \mu \vec{1} \oslash s$}
  \State ${\rm iter} \gets 0$
  \While{${\rm iter} < {\rm iter}_{\max}$}
  \State $\ICflag \gets \texttt{False}$
  \State compute $\nabla_p L$ and $E_{\mu}(p,s,y,z)$
  \State \textbf{if} {$E_{\mu}(p,s,v,y,z) \leq \epsilon_{\rm tol}$} \textbf{then break} 
  \State compute $D_{\mathcal{E}}$, $D_{\mathcal{I}}$, $\Sigma$, $c_{\mathcal{I}}(p)$ from \eqref{eq:matrix_blocks} and $\nabla^2_{pp}L$
  \State $S \gets 0$
  \State compute $J'$ from \eqref{eq:newton_step_corrected} \label{alg:correction}
  \State compute $\Delta p$ $\Delta s$, $\Delta v$, $\Delta y$ and $\Delta z$ by solving \eqref{eq:newton_step_permuted} with a sparse routine
  \State compute $\nu^+$ from \eqref{eq:nu_update} and set $\nu \gets \nu^+$
  \State $\gamma^M \gets$ perform line search until \eqref{eq:line_search_conds} is satisfied or \eqref{eq:cond_sucess} is violated
  \If {$\gamma^M \leq \epsilon_{\rm ls}$ {\bf and} \ICflag}
     \State report ``failed after inertia correction''
     \State \textbf{break}
  \ElsIf {$\gamma^M \leq \epsilon_{\rm ls}$}  
    \State compute $\Ssub{\mathcal{E}}$ from \eqref{eq:S_E}, $\Ssub{\mathcal{I}}$ from \eqref{eq:S_I} and $\delta_S$ from \eqref{eq:delta_S}
    \State $S \gets \Ssub{\mathcal{E}} + \Ssub{\mathcal{I}} + \delta_S I$
    \State $\ICflag \gets \texttt{True}$
    \State \textbf{go to} \ref{alg:correction}:
  \EndIf
  \State compute $p^+$, $v^+$, $y^+$, $z^+$, $s^+$ from \eqref{eq:line_search_update}
  \State $(p, s, v, y, z, {\rm iter}) \gets (p^+, s^+, v^+, y^+, z^+, {\rm iter}+1)$
  \EndWhile
  \State \textbf{return} $p$, $s$, $v$, $y$, $z$
  \EndProcedure
 \end{algorithmic} 
\end{algorithm}

\subsection*{Appendix C: Proof of Theorem 1}
\begin{proof}[Proof]
For brevity, we omit the dependence of $b_j(p)$ and $q_j(p)$ on $p$. Notice that $D_{\mathcal{E}}$ can be written as
\begin{equation}
D_{\mathcal{E}} = \left[\begin{matrix}
b^T_1 + b^T_2 & -b^T_2 &  &   \\
-b^T_2 & \ddots & \ddots &   \\
 &  \ddots & \ddots & -b^T_K \\
 &   & -b^T_K & b^T_K + b^T_{K+1}
\end{matrix}\right].
\end{equation}
Define, for any $k \in \mathbb{N}$, the following matrices:
\begin{subequations}
\begin{align}
L_k &= \left[ {\begin{matrix}
   1 & & &  \\   
   1 & 1 & &  \\
   \vdots & \ddots & \ddots &  \\
   1 & \cdots & 1 & 1 \\
\end{matrix}} \right] \in \mathbb{R}^{k \times k}, \\
U_k &= \left[ {\begin{matrix}
   1 & -1 & &  \\   
     & 1 & \ddots &  \\
     & & \ddots & -1 \\
     & & & 1 \\
\end{matrix}} \right] \in \mathbb{R}^{k \times k}, 
\end{align}
\end{subequations}
and, for $k = 2, \ldots, K$, define
\begin{subequations}
\begin{align}
D_k &= {\rm diag}\left([q^T_j]_{j \in \{1,\ldots,K+1\} \setminus \{k\}}\right)^T \in \mathbb{R}^{2(g+n)K \times K}, \\
M_k &= \left[ {\begin{matrix}
   L_{k-1} \otimes I_{2g} & 0 \\
   0 & L^T_{K-k+1} \otimes I_{2g}
\end{matrix}} \right] \in \mathbb{R}^{2gK \times 2gK}, 
\end{align}
\end{subequations}
where $\otimes$ denotes the Kronecker product and $I_k$ denotes the $k \times k$ identity matrix. Lastly, define
\begin{equation}
Q = \left([q^T_j]^{K+1}_{j=1}\right)^T,
\end{equation}
and let $e_k \in \mathbb{R}^k$ be the last (rightmost) column of the $k \times k$ identity matrix. Then, for any $k=2,\ldots,K$, we have by direct computation that

{\setlength{\parskip}{-1em}
\small
\begin{align}
&(D_p c_{\mathcal{E}}) \cdot (M_k D_k) = \nonumber \\
 &\left[ {\begin{matrix}
   U_{k-1} + e_{k-1}(b^T_k (Q)_{1:k-1,:}) & -e_{k-1}(b^T_k (Q)_{k+1:K+1,:}) \\
   -e_{k-1}(b^T_k (Q)_{1:k-1,:}) & U^T_{K-k+1} + e_{k-1}(b^T_k (Q)_{k+1:K+1,:})
\end{matrix}} \right].
\end{align}
}%
We next eliminate the off-diagonal entries of $U_{k-1}$ and $U^T_{K-k+1}$ using elementary column operations. Thus, there exists an invertible matrix $C \in \mathbb{R}^{K \times K}$ such that
\begin{align}
&(D_p c_{\mathcal{E}}) \cdot (M_k D_k) C = \nonumber \\
 &\left[ {\begin{matrix}
 I_{k-2} & & & \\
 \boxtimes & 1+\sum^{k-1}_{j=1}{b^T_k q_j} & -\sum^{K+1}_{j=k+1}{b^T_k q_j} & \boxtimes \\
 \boxtimes & -\sum^{k-1}_{j=1}{b^T_k q_j} & 1+\sum^{K+1}_{j=k+1}{b^T_k q_j} & \boxtimes \\
  & & & I_{K-k}
\end{matrix}} \right],
\end{align}
where the symbol $\boxtimes$ denotes blocks with possibly non-zero, but unimportant entries. We can eliminate the $\boxtimes$ entries using elementary row operations, so there exists an invertible matrix $R \in \mathbb{R}^{K \times K}$ such that
\begin{align}
&A_k = R(D_p c_{\mathcal{E}}) \cdot (M_k D_k) C = \nonumber \\
 &\quad \left[ {\begin{matrix}
 I_{k-2} & & & \\
  & 1+\sum^{k-1}_{j=1}{b^T_k q_j} & -\sum^{K+1}_{j=k+1}{b^T_k q_j} &  \\
  & -\sum^{k-1}_{j=1}{b^T_k q_j} & 1+\sum^{K+1}_{j=k+1}{b^T_k q_j} &  \\
  & & & I_{K-k}
\end{matrix}} \right],
\end{align}
where we named the final matrix as $A_k$ for convenience. The determinant of $A_k$ can be readily compued as
\begin{subequations}
\begin{align}
\det(A_k) &= \left(1+{\textstyle\sum}^{k-1}_{j=1}{b^T_k q_j}\right) \left(1+{\textstyle\sum}^{K+1}_{j=k+1}{b^T_k q_j}\right) \nonumber \\
&\quad - \left({\textstyle\sum}^{k-1}_{j=1}{b^T_k q_j}\right) \left({\textstyle\sum}^{K+1}_{j=k+1}{b^T_k q_j}\right), \\
\det(A_k) &= 1 + b^T_k \left({\textstyle\sum}^{k-1}_{j=1}{ q_j} +{\textstyle\sum}^{K+1}_{j=k+1}{q_j}\right).
\end{align}
\end{subequations}
Notice that $b^T_k q_k = 1$, and hence
\begin{equation}
\label{eq:det_Ak}
\det(A_k) = b^T_k {\textstyle\sum}^{K+1}_{j=1}{q_j} = b^T_k Q {\vec 1},
\end{equation}
where ${\vec 1}$ denotes a vector with all entries equal to $1$. The previous argument (with small modifications) still holds for $k=1$ and $k=K+1$, so \eqref{eq:det_Ak} holds for $k=1,\ldots,K+1$. Assume for contradiction that $D_{\mathcal{E}}$ is rank deficient, then from properties of the rank (see \cite{linear_algebra}), we have that
\begin{equation}
{\rm rank}(A_k) \leq {\rm rank}(D_{\mathcal{E}}) < K,
\end{equation}
so $A_k$ is singular and therefore
\begin{equation}
q^T_k Q {\vec 1} = (b^T_k b_k)^{-1} b^T_k Q {\vec 1} = (b^T_k b_k)^{-1} \det(A_k) = 0,
\end{equation}
for $k=1,\ldots,K+1$. In matrix-vector form we have that
\begin{equation}
Q^T Q {\vec 1} = 0,
\end{equation}
so
\begin{equation}
\left({\textstyle\sum}^{K+1}_{j=1} q_j\right)^2 = \left(Q {\vec 1}\right)^T \left(Q {\vec 1}\right) = {\vec 1}^T \left(Q^T Q {\vec 1}\right) = 0,
\end{equation}
which implies that ${\textstyle\sum}^{K+1}_{j=1} q_j = 0$, and so we have a contradiction.
\end{proof}

\subsection*{Appendix D: Proof of Theorem 2}
\begin{proof}[Proof]
For brevity, we define the scalars $a_j$ as
\begin{equation}
a_j = \left\{ {\begin{array}{ll}
   0, & j=0 \\
   2 w_j ((y)_j - (y)_{j-1}) & j \in \{1,\ldots, K+1\}
\end{array}} \right. ,
\end{equation}
and $\lambda^-$ as
\begin{equation}
\lambda^- = \min_{k=0,\ldots,K+1} a_k \leq 0,
\end{equation}
so $l_{\mathcal{E}}$ becomes
\begin{equation}
l_{\mathcal{E}} = -2 \lambda^- \left(1 + \cos \left(\frac{\pi}{K+1}\right) \right) ,
\end{equation}
and we can write $\nabla^2_{pp} (y^T c_{\mathcal{E}})$ as
\begin{equation}
\nabla^2_{pp} (y^T c_{\mathcal{E}}) = \left[\begin{matrix}
(a_1 + a_2)I' & -a_2 I' &  &   \\
-a_2 I' & \ddots & \ddots &   \\
 &  \ddots & \ddots & -a_K I' \\
 &   & -a_K I' & (a_K + a_{K+1}) I'
\end{matrix}\right],
\end{equation}
where $I'$ is the matrix defined in \eqref{eq:Iprime}. Let $I''$ be the $2(g+n) \times 2(g+n)$ identity matrix, and let us define the following matrices: 
\begin{subequations}
\begin{align}
D &= \left[ {\begin{matrix}
   a_1 I' & 0 & \cdots & 0 \\
   0 & \ddots & \ddots & \vdots \\
   \vdots & \ddots & \ddots & 0 \\
   0 & \cdots & 0 & a_{K+1} I' \\
\end{matrix}} \right] \in \mathbb{R}^{2(g+n)(K+1) \times 2(g+n)(K+1)}, \\
A &= \left[ {\begin{matrix}
   I'' & 0 & \cdots &  & 0 \\   
   -I'' & I'' & 0 & \cdots & 0 \\
   0 & \ddots & \ddots & \ddots & \vdots \\
   \vdots & \ddots & \ddots & \ddots & 0 \\
   0 & \cdots & 0 & -I'' & I'' \\
   0 & \cdots &  & 0 & -I'' \\
\end{matrix}} \right] \in \mathbb{R}^{2(g+n)(K+1) \times 2(g+n)K}.
\end{align}
\end{subequations}
We can then factor $\nabla^2_{pp} (y^T c_{\mathcal{E}})$ as
\begin{equation}
\nabla^2_{pp} (y^T c_{\mathcal{E}}) = A^T D A.
\end{equation}
Let the following matrices be:
\begin{subequations}
\begin{align}
D_1 &= \left[ {\begin{matrix}
    (a_1 - \lambda^-) I' & 0 & \cdots & 0 \\
   0 & \ddots & \ddots & \vdots \\
   \vdots & \ddots & \ddots & 0 \\
   0 & \cdots & 0 & (a_{K+1} - \lambda^-) I' \\
\end{matrix}} \right], \\
D_2 &= \left[ {\begin{matrix}
    I' & 0 & \cdots & 0 \\
   0 & \ddots & \ddots & \vdots \\
   \vdots & \ddots & \ddots & 0 \\
   0 & \cdots & 0 & I' \\
\end{matrix}} \right] \in \mathbb{R}^{2(g+n)(K+1) \times 2(g+n)(K+1)},
\end{align}
\end{subequations}
then clearly
\begin{equation}
\nabla^2_{pp} (y^T c_{\mathcal{E}}) = \lambda^- A^T D_2 A + A^T D_1 A.
\end{equation}
The matrix $D_1$ is block diagonal with positive semidefinite blocks, so $D_1 \succeq 0$. This means that $A^T D_1 A \succeq 0$ and thus, from the concavity of the smallest eigenvalue (see \cite{matrix_calculus}), we have that
\begin{align}\label{eq:lmin_ytce}
\lambda_{\min}\left({\nabla^2_{pp} (y^T c_{\mathcal{E}})}\right) &\geq \lambda_{\min}\left({\lambda^- A^T D_2 A}\right) + \lambda_{\min}\left({A^T D_1 A}\right) \nonumber \\
&\geq \lambda_{\min}\left({\lambda^- A^T D_2 A}\right).
\end{align}
Denote the Kronecker product by $\otimes$, then
\begin{equation}
A^T D_2 A = \left[ {\begin{matrix}
   2I' & -I' & 0 & \cdots & 0 \\
   -I' & 2I' & \ddots & \ddots & \vdots \\
   0 & \ddots & \ddots & \ddots & 0 \\
   \vdots & \ddots & \ddots & 2I' & -I' \\
   0 & \cdots & 0 & -I' & 2I' \\
\end{matrix}} \right] = Y \otimes I',
\end{equation}
where
\begin{equation}
Y = \left[ {\begin{matrix}
   2 & -1 & 0 & \cdots & 0 \\
   -1 & 2 & \ddots & \ddots & \vdots \\
   0 & \ddots & \ddots & \ddots & 0 \\
   \vdots & \ddots & \ddots & 2 & -1 \\
   0 & \cdots & 0 & -1 & 2 \\
\end{matrix}} \right] \in \mathbb{R}^{K \times K}.
\end{equation}
Denote the eigenvalues of $Y$ by $\xi_k$ for $k=1,\ldots,K$, then (see \cite{kulkarni})
\begin{equation}
\xi_k = 2 - 2 \cos\left(\frac{k \pi}{K + 1}\right).
\end{equation}
The eigenvalues of the Kronecker product correspond to the product of the eigenvalues of each matrix, so the eigenvalues of $Y \otimes I$ are $\xi_k$ for $k=1\ldots,K$, each repeated $2g$ times, and $0$ repeated $2nK$ times \cite{eigenvalues}. As $\lambda^- \leq 0$ we can write \eqref{eq:lmin_ytce} as
\begin{subequations}
\begin{align}
\lambda_{\min}\left({\nabla^2_{pp} (y^T c_{\mathcal{E}})}\right) &\geq \lambda^- \cdot \lambda_{\max}\left({A^T D_2 A}\right), \\
\lambda_{\min}\left({\nabla^2_{pp} (y^T c_{\mathcal{E}})}\right) &\geq \lambda^- \cdot \lambda_{\max}\left({Y \otimes I'}\right), \\
\lambda_{\min}\left({\nabla^2_{pp} (y^T c_{\mathcal{E}})}\right) &\geq \lambda^- \cdot \max \left\{ \lambda_{\max}\left({Y}\right), 0 \right\}, \\
\lambda_{\min}\left({\nabla^2_{pp} (y^T c_{\mathcal{E}})}\right) &\geq \lambda^-  \left(2 - 2 \cos\left(\frac{K \pi}{K + 1}\right)\right), \\
\lambda_{\min}\left({\nabla^2_{pp} (y^T c_{\mathcal{E}})}\right) &\geq 2 \lambda^- \left(1 + \cos\left(\frac{ \pi}{K + 1}\right)\right) = -l_{\mathcal{E}},
\end{align}
\end{subequations}
and the claim follows.
\end{proof}
\fi


%

\ifCLASSOPTIONcaptionsoff
  \newpage
\fi

\end{document}